\documentclass[preprint,11pt]{amsart}

\usepackage{geometry}
\usepackage[T1]{fontenc}
\usepackage[utf8]{inputenc}
\usepackage[english]{babel}
\usepackage{amsmath,amsthm,amssymb,mathrsfs}
\usepackage{cancel}
\newcommand{\abs}[1]{\left\lvert #1 \right\rvert}
\numberwithin{equation}{section}
\usepackage{booktabs}
\usepackage{subcaption}
\usepackage[all]{xy}
\usepackage{tikz-cd, multirow}
\usepackage{enumitem}
\usepackage{array}
\usepackage{graphicx}
\usepackage{makecell}
\usepackage{float}
\usepackage{mathtools}
\usepackage[english]{varioref}
\usepackage{appendix}
\usepackage{multirow}
\usepackage{xcolor}
\usepackage{enumitem}
\usepackage{fancyhdr}
\usepackage{hyperref}
\usepackage{longtable}
\usepackage{lipsum}
\usepackage{chngcntr}
\usepackage[absolute,overlay]{textpos}
\usepackage{lipsum}
\usepackage{setspace}
\usepackage[justification=centering]{caption}
\counterwithin{figure}{section}
\hypersetup{ bookmarks=true, colorlinks=true, linkcolor=red, citecolor=blue, urlcolor=cyan}

\newtheorem{theorem}{Theorem}[section]
\newtheorem{lemma}[theorem]{Lemma}
\newtheorem{corollary}[theorem]{Corollary}
\newtheorem{proposition}[theorem]{Proposition}

\newcounter{maintheorem}                       

\newtheorem{mainth}[maintheorem]{Theorem}
\theoremstyle{remark}
\newtheorem{remark}[theorem]{Remark}

\theoremstyle{definition}
\newtheorem{definition}[theorem]{Definition}

\newcounter{smallromans}

\newenvironment{romanenumerate}
{\begin{list}{{\normalfont\textrm{(\roman{smallromans})}}}%
  {\usecounter{smallromans}\setlength{\itemindent}{0cm}%
   \setlength{\leftmargin}{5.5ex}\setlength{\labelwidth}{5.5ex}%
   \setlength{\topsep}{.5ex}\setlength{\partopsep}{.5ex}%
   \setlength{\itemsep}{0.1ex}}}%
{\end{list}}
\def\Circlearrowright{\ensuremath{\rotatebox[origin=c]{180}{$\circlearrowright$}}}
\def\Circlearrowleft{\ensuremath{\rotatebox[origin=c]{180}{$\circlearrowleft$}}}

\newcommand{\cM}{{\mathcal M}}

\newcommand{\cK}{{\mathcal K}}
\newcommand{\cT}{{\mathcal T}}
\newcommand{\cF}{{\mathcal F}} 
\newcommand{\bA}{{\mathbb A}}
\newcommand{\bP}{{\mathbb P}}
\newcommand{\bC}{{\mathbb C}}
\newcommand{\bR}{{\mathbb R}}
\newcommand{\bZ}{{\mathbb Z}}
\newcommand{\bQ}{{\mathbb Q}}

\newcommand{\fD}{{\mathfrak D}}

\newcommand{\ul}[1]{\underline{#1}}
\newcommand{\ol}[1]{\overline{#1}}

\allowdisplaybreaks

\makeatletter
\def\keywords{\xdef\@thefnmark{}\@footnotetext}
\makeatother

\begin{document}

\onehalfspacing

\title{The classification of automorphisms and quotients \\of Calabi-Yau threefolds of type $A$}

\author{Martina Monti}
\address{Dipartimento di Matematica Federigo Enriques - UNIMI, Via Cesare Saldini, 50, 20133 Milano MI, Italia and  Laboratoire de Mathématiques et Applications - Université de Poitiers, UMR 7348 du CNRS, 11 bd Marie et Pierre Curie, 86073 Poitiers Cedex 9, Frances}
\email{martina.monti1@unimi.it}

\begin{abstract}
The aim of the paper is to investigate
the only two families $\cF^A_{G}$ of Calabi-Yau $3$-folds $A/G$ with $A$ an abelian $3$-fold and $G\le \text{Aut}(A)$ a finite group acting freely: one in constructed in \cite{CD2} and the other is presented here.
We provide a complete classification of the automorphism group of $X\in \cF^A_{G}$. Additionally, we construct and classify the quotients $X/\Upsilon$ for any $\Upsilon\le \text{Aut}(X)$. Specifically, for those groups $\Upsilon$ that preserve the volume form of $X$ then $X/\Upsilon$ admits a desingularization $Y$ which is a Calabi-Yau $3$-fold: we compute the Hodge numbers and the fundamental group of these $Y$, thereby determining all topological in-equivalent Calabi-Yau $3$-folds obtained in this way. 

\end{abstract}


\maketitle
\section*{Introduction}
In the last years free quotients of complex tori are not again complex tori, called \emph{Generalized Hyperelliptic Manifolds (GHM)},
have gained significant attention  as they are the natural generalizations of the bi-elliptic surfaces, see \cite{La}, \cite{CC}, \cite{CD2}, \cite{OS}, \cite{HK}, \cite{HL19}, \cite{DG23},  \cite{Cat23}.

We are interested in GHMs that have trivial canonical bundle, specifically in those that Calabi-Yau manifolds. The GHMs that are also Calabi-Yau manifolds were first introduced by Oguiso and Sakurai in \cite{OS}. As the authors observed these Calabi-Yau manifolds have infinite fundamental group, $c_2=0$ and they don't contain rational curves.
Moreover, we observe that they appear only in odd dimension higher than $1$ and the complex torus in always projective, in fact they are called \emph{Calabi-Yau manifolds of type $A$}, see lemma \ref{definitionCYA}.
In \cite{OS} the authors provided a full classification of these manifolds in dimension $3$. There are only two possible groups to construct Calabi-Yau threefolds of type $A$: the group $(\bZ/2\bZ)^2$ and the dihedral group $\fD_4$ of order $8$. In this article we consider the two families $\cF^A_G$ of Calabi-Yau threefolds of type $A$. The main purpose is to classify the automorphism group and the quotients of $A/G=X\in \cF^A_G$ and we mainly focus on the case $G\simeq \fD_4$.

The family $\cF^A_{\fD_4}$ is constructed in \cite{CD2} and using similar techniques we construct the $\cF^A_{(\bZ/2\bZ)^2}$, see Section \ref{section8}. Both the families are irreducible and the space of parameters has dimension equal to $h^{2,1}(X)=h^{1,1}(X)=\begin{cases} 2 & G\simeq \fD_4 \\
3 & G\simeq (\bZ/2\bZ)^2 \end{cases}$ with $X\in \cF^A_{G}$. 
According to a well-known result, see Theorem \ref{aut GHV}, the automorphism group of $X\in \cF^A_{G}$ can be computed by study the normalizer $\text{N}_{\text{Aut}(A)}(G)$. We find that $\text{Aut}(X)$ is finite and contains only elements of order $2$ for every $X\in \cF^A_G$. In particular, we explicitly describe the generators: for $G\simeq \fD_4$ 
see Theorem \ref{mainth1} and for $G\simeq (\bZ/2\bZ)^2$ see Theorem \ref{mainth2}. This knowledge allows us to construct and classify all the possible quotients $Y$ of $X\in \cF^A_{G}$, up to a desingularization.

\begin{mainth}[Theorem \ref{mainth2} and Theorem \ref{mainth8}]
    Let $X\in \cF^A_{G}$ and $\Upsilon\le \textup{Aut}(X)$.
    \begin{romanenumerate}
        \item If $G\simeq \fD_4$ then each quotient $X/\Upsilon$ admits a crepant resolution $\beta\colon Y \rightarrow X/\Upsilon$ where $Y$ is a Calabi-Yau $3$-fold.

        \item If $G\simeq(\bZ/2\bZ)^2$ Let $\beta\colon  Y\rightarrow X/\Upsilon$ be the blow up of the singular locus of $X/\Upsilon$. The followings hold.
        \begin{enumerate}
            \item[a.] If $\Upsilon$ preserves the volume form of $X$, $\beta$ is a crepant resolution and $Y$ is a Calabi-Yau $3$-fold. 
            
            \item[b.] If there exists at least one $\alpha_X \in \Upsilon$ that fixes surfaces on $X$ then $Y$ has negative Kodaira dimension.
            
            \item[c.]  Otherwise, $Y$ has trivial Kodaira dimension and its canonical bundle is not trivial.
            \end{enumerate}
    \end{romanenumerate}    
\end{mainth}

 We are particular interested in quotients $Y$ of $X$ which, up to a resolution of singularities, are Calabi-Yau threefolds. In fact our aim is to classify them. To do this we compute the topological invariants $h^{1,1}(Y)$, $h^{2,1}(Y)$ and $\pi_1(Y)$.
For the case $G\simeq \fD_4$, we summarize our results from Table \ref{table3} to Table \ref{table14}: we find at least $19$ non-homeomorphic Calabi-Yau threefolds constructed in this way.
For the case $G\simeq(\bZ/\bZ)^2$, we compute their Hodge numbers in Proposition \ref{hdgnumbers} and we refer to \cite[Section 1]{DW} for a description of their the fundamental groups.
Thanks to the investigation above, we prove that there exist free quotients of $X\in \cF^A_{G}$ which belong to the same family of $X$. This allow us to define a map from $\cF^A_{G}$ to itself which tell us that we can move in $\cF^A_{G}$ by using quotients and coverings: we analyze the case $G\simeq \fD_4$ obtaining a map of degree $2$, see Section \ref{sec: map on family}.

We explain more in details the structure of the article.
In Section \ref{sec: CYtypeA} we introduce the Calabi-Yau manifolds of type $A$ and we collect some of their properties.
In Section \ref{sec: CY 3fold} we give some preliminaries on Calabi-Yau $3$-folds, specifically on involution on it which will be useful in the subsequent sections. The central part of the article, from Section \ref{sec: familyd4} to Section \ref{sec: map on family}, is dedicated to the study of the family $\cF^A_{\fD_4}$. In the last section we briefly present the results for the the family $\cF^A_{(\bZ/2\bZ)^2}$ by extending the techniques used in the previous sections.

\textbf{Acknowledgements.} This work is part of my PhD thesis and I would like to thank my advisors: Alice Garbagnati for introducing me the subject, for her unrivalled assistance, si\-gnificant discussions and suggestions throughout the redaction; and Alessandra Sarti for her hints during the preparation of this paper.
I would like to thank also Bert van Geemen for his help in the computations of Section \ref{sec: neronseveri}. Moreover, I would like to thank Cécilie Gachet for helpful discussions during my visiting at the “Université Côte d'Azur”. Finally, I would like to thank the Departments of Mathematics of “Università degli Studi di Milano” and “Université de Poitiers” for the fundings; I am also partially funded by the grant “VINCI 2022-C2-62” issued by the “Université Franco-Italienne”.\vspace*{1cm}

\section{Preliminaries on Calabi-Yau manifolds of type $A$}\label{sec: CYtypeA}
\begin{definition}\label{def: GHV}
A \emph{Generalized Hyperelliptic Manifold} (GHM) $Y$ is a manifold not isomorphic to a complex torus but admitting a complex torus $T$ as finite étale Galois cover. 
\end{definition}

\begin{remark}\label{notranslation}
Let $Y=T/G$ be a GHM. By \cite[Remark 1.7]{CD20} we can assume in Definition \ref{def: GHV} that $G$ does not contain translations.
\end{remark}

\begin{definition}\label{def: GHV}
A group $\Gamma \le \textup{Iso}(\bR^n)$ is said to be a \emph{crystallographic group} if it is discrete and $\bR^n/\Gamma$ is compact.
\end{definition}

\begin{remark}\label{exact-sequence}
According to the first theorem of Bieberbach \cite[Theorem 3.2.1]{Wolf11} any $\Gamma$ crystallographic group is characterized by the
    following exact sequence:
   \begin{equation}\label{crystallographic sequence}
   0\longrightarrow \Lambda' \xrightarrow{\quad i\quad} \Gamma \xrightarrow{\quad l\quad} G' \longrightarrow 0
\end{equation}
where $i(\lambda) = (I,\lambda)$ and $l(M,m) = M$
for every $\lambda\in \Lambda'$ and $(M,m)\in \Gamma$. In particular, $\Lambda'=\Gamma\cap \bR^n$ is the maximal abelian and normal subgroup of finite index in $\Gamma$.
\end{remark}

By using standard result of topological covering and the definition of crystallographic group we obtain the following. 

\begin{lemma}\label{fund_group}
    Let $Y=T/G$ be a GHM. Then there exists $\Gamma \le \textup{Iso}(\bC^n) $ such that $\pi_1(Y)\simeq \Gamma$ and $Y=\bC^n/\Gamma$. In particular, $\Gamma$ is a crystallographic group. 
\end{lemma}

\begin{remark}
By \cite[Corollary 4.5]{DG23} we know that the GHM $T/G$ is topologically determined by the group $G$ which acts freely on $T$ and does not contain translations, that is, $G$ characterizes the fundamental group of $T/G$. Hence we may say that $T/G$ is a \emph{GHM with the group $G$}.
\end{remark}

The following result characterizes the automorphism group of a GHM.

\begin{corollary}\label{aut GHV} \textup{\cite[Theorem 4.2]{HL19}}
Let $Y=T/G$ be a \textup{GHM} with the group $G$. Then we have the following epimorphism of groups:
\begin{center}
    \begin{tikzcd}
     \varphi\colon  \textup{N}_{\textup{Aut}(T)}(G)\arrow[r,"\simeq"] &   \textup{Aut}(Y) \\
      \alpha_T \arrow[r,mapsto] & \alpha_Y
    \end{tikzcd}
\end{center}
such that \textup{ker}$(\varphi)=G$ and $\alpha_Y$ makes the following diagram commutative
        \begin{tikzcd}
            T\arrow[r,"\alpha_T"] \arrow[d,"\pi"']& T \arrow[d,"\pi"] \\
            Y \arrow[r, "\alpha_Y"'] & Y
        \end{tikzcd}.
\end{corollary}

\begin{remark}
In \cite[pag. 10 (a)]{Beau83b} the author have observed that given a free quotient $Y=(T\times S)/G$ with $T$ complex torus, $S$ K\"ahler manifolds with $b_1(S)=0$ and $G$ finite group; then the automorphism group $\text{Aut}(Y)$ can be identified with the normalizer of $G$ in $\text{Aut}(T)\times \text{Aut}(M)$
Thus Corollary \ref{aut GHV} is essentially a particular case of the observation above.
\end{remark}

\begin{definition}
    With the notation of Corollary \ref{aut GHV}. We call $\alpha_Y$ the \emph{automorphism induced by $\alpha_T$}.
\end{definition}

\begin{definition}\label{def: CY}
A \emph{Calabi-Yau $n$-fold} $Y$ is a compact, K\"ahler, complex $n$-fold such that the canonical bundle $\cK_Y$ is trivial and $h^{i,0}(Y)=0$ if $i\not=0,n$.
\end{definition}

Let $Y$ be a Calabi-Yau $n$-fold: we call  \emph{period} or \emph{volume form} of $Y$ the generator $\omega_Y$ of $H^{n,0}(Y)$. We recall that \begin{equation}\label{chi_CY}
    \chi(Y)=\begin{cases}
        0 & \text{n is odd} \\
        2 & \text{n is even}.
    \end{cases}
\end{equation} 

\begin{lemma}\label{definitionCYA}
Let $Y=T/G$ be a \textup{GHM} $n$-fold with the group $G$. If it is a Calabi-Yau $n$-fold, then $n$ is odd, $n\not=1$ and the complex torus $T$ has an ample line bundle.
\end{lemma}
\begin{proof}
Let us denote  $\pi: T \longrightarrow Y=T/G$ the  finite étale Galois covering. According to the formula $\chi(T)=\text{deg}(\pi) \chi(Y)$, see \cite[Proposition 1.1.28 ]{Laz}, and since $\chi(T)=0$, we obtain $\chi(Y)=0$. Thus by \eqref{chi_CY}, $n$ must be odd. Moreover, by definition the GHMs don't exist in dimension $1$ hence $n\ge3$.
Let $Y$ be a Calabi-Yau GHM of dimension $n\ge 3$.  We recall, for the reader convenience, that every Calabi-Yau $n$-fold with $n>2$ is projective (\cite[Corollary 5.3.3. and Example 5.3.6 (ii)]{Hu} ). Thus there exists an ample line bundle $L$ on $Y$. 
Since $\pi: T \longrightarrow Y=T/G$  is a finite morphism, $\pi^\ast (L)$ defines an ample line bundle on $T$ by \cite[Proposition 4.3]{Hart2}. Hence, $T$ is an abelian variety.
\end{proof}

\begin{definition}
A \emph{Calabi-Yau $(2n+1)$-fold of type $A$} is a GHM which is also a Calabi-Yau manifold.
\end{definition}

\section{Preliminaries on Calabi-Yau $3$-folds} \label{sec: CY 3fold}
Let $Y$ be a Calabi-Yau $3$-fold. We have:\begin{equation}\label{Euler of CY}e(Y)=2\{h^{1,1}(Y)-h^{2,1}(Y)\}. \end{equation}

\begin{lemma}\label{Picgroup}
    Let $Y$ a Calabi-Yau $3$-fold. Then $\textup{Pic}(Y)\simeq \bZ^{h^{1,1}(Y)}\oplus \bA\textup{b}(\pi_1(Y))$
    where $\bA\textup{b}(\pi_1(Y))$ is the abelianization of $\pi_1(Y)$.
\end{lemma} 
\begin{proof}
    Since $H^{1,0}(Y)=H^{2,0}(Y)=0$, the first Chern map is an isomorphism: $\text{Pic}(Y)\simeq H^2(X,\bZ)$. Combining the Universal Coefficient Theorem \cite[Section 3.1]{AT}
and the fact that $H_1(Y,\bZ)\simeq \bA\text{b}(\pi_1(Y))$ by \cite[Section 2.A]{AT}, we obtain the result.

\end{proof}

By the Bogomolov-Tian-Todorov unobstructed theorem \cite[Theorem 14.10]{GHJ} and \cite[Remark 14.11]{GHJ}, we have the following characterization of the space that parametrizes the first order deformations of a Calabi-Yau threefolds.

\begin{theorem}\label{deform_CY3}
Let $Y$ be a Calabi-Yau threefold. Then the space of deformations of $Y$ is unobstructed and parametrized by $H^1(Y,\cT_Y)\simeq H^{2,1}(Y)$ where $\cT_Y$ is the tangent sheaf.
\end{theorem}

We collect some results on involutions on Calabi-Yau threefolds which will be useful later.

\begin{definition}
Let $f\colon Y'\longrightarrow Y$ be a resolution of singularities. It is called \emph{crepant resolution} if $\cK_{Y'}=f^\ast \cK_{Y}$, \emph{i.e.} $f$ preserves the canonical class of $Y'$.
\end{definition}

\begin{proposition}\label{crep-res-3}\textup{\cite[Theorem 1]{Roan96}}
The quotients $\bC^3/G$ with $G\le \textup{SL}(3,\bC)$ a finite group admits a crepant resolution. 
\end{proposition}

\begin{proposition}\label{quotientCYthreefold} 
Let $Y$ be a Calabi-Yau $3$-fold and $\alpha_Y$ be an involution on $Y$ which preserves $\omega_Y$. The fixed locus of $\alpha_Y$ is either empty or made up of a finite number of disjoint smooth curves. Moreover, the quotient variety ${Y}/{\alpha_Y}$ is birational to a Calabi-Yau $3$-fold.
\end{proposition}
\begin{proof}
Let $y\in \text{Fix}(\alpha_Y)\not=\emptyset$.
Since $\alpha_Y$ preserves the volume form of $Y$ and by \cite[Lemma 1]{Car}, there exists a local holomorphic coordinate system $z_1, z_2, z_3$ around $y$ such that
$\alpha_Y$ linearizes into $\delta=\text{diag}(-1,-1,1)$. Hence the fixed locus of $\alpha_Y$ is a disjoint union of smooth curves and since it is compact there are a finite numbers of curves in it.
Assume Fix$(\alpha_Y)=\emptyset$ then it is easy to see that $Y/\alpha_Y$ is a Calabi-Yau $3$-fold. 
If Fix$(\alpha_Y)\not=\emptyset$: by applying Proposition \ref{crep-res-3} there exist a  resolution of singularities 
$\beta: Z \longrightarrow Y/\alpha_Y$  such that $K_Z=0$. Since $h^{j,0}(X)=h^{j,0}(X/\alpha_Y)=0$ and they are birational invariants we get $h^{j,0}(Z)=0$ for $j=1,2$. Thus $Z$ is a Calabi-Yau $3$-fold.
\end{proof}

\begin{proposition}\label{kodairadimension}
Let $Y$ be Calabi-Yau $3$-folds and $\alpha_Y$ be an involution on $Y$ which does not preserve $\omega_Y$. Then $\textup{Fix}(\alpha_Y)\not=\emptyset$ and it is the disjoint union of odd-codimensional submanifolds of $Y$. Moreover, \begin{romanenumerate}
    \item If $\alpha_Y$ fixes only points then $Y/\alpha_Y$ admits a desingularization $Z$ such that $k(Z)=0$ but $Z$ has no trivial canonical bundle.
    \item Otherwise, $Y/\alpha_Y$ has negative Kodaira dimension.
\end{romanenumerate}
\end{proposition}
\begin{proof}
Let us assume that $\alpha_Y$ has trivial fixed locus. Then $\chi(Y)=2\chi(Y/\alpha_Z)$ and since $\chi(Y)=0$ we have $\chi(Y/\alpha_Z)=0$. This leads to a contradiction since we find $h^{3,0}(Y/\alpha_Z)=1$ which means that $\alpha_Y$ preserves $\omega_Y$. Thus $\text{Fix}(\alpha_Y)\not=\emptyset$. 
Let $y\in \text{Fix}(\alpha_Y)$:
according to \cite[Lemma 1]{Car} there exists a neighbourhood of $y$ where $\alpha_Y$ can be diagonalized in $\delta_y$. Since $\delta_y$ is an involution and $\delta_y\not\in \text{SL}(3,\bC)$ then the eigenvalue $-1$ has odd multiplicity.
Therefore, the fixed locus of $\alpha_Y$ consists of a finite number of smooth manifolds of odd-codimension. Let $F_{2i}$ the $2i$-dimension subset of $\text{Fix}(\alpha_Y)$. We consider the following diagram.
\begin{center}
\begin{tikzcd}
\widetilde{\alpha_Y}\Circlearrowright\widetilde{Y} \arrow[r, "\gamma"] \arrow[d, "\pi", "2:1"'] & Y \arrow[d, "2:1"]\\
Z:=\widetilde{Y}/\langle \widetilde{\alpha_Y}\rangle \simeq \widetilde{Y/\langle\alpha_Y\rangle} \arrow[r, "\beta"'] & Y/\langle\alpha_Y\rangle
\end{tikzcd}
\end{center}
where $\gamma$ blows up one times each irreducible component $F_{2i}$ for $0\le i\le 1$, since $\alpha_Y$ preserves the blown up locus it lifts to an involution $\widetilde{\alpha_Y}$ on $\widetilde{Y}$ and $\beta$ is the blow up of the singular locus of $Y/\langle \alpha_Y \rangle$. It is easy to check that the diagram commutes.
The ramification divisor of ${\pi}$ is given by 
$$R_{0} + \gamma^{-1}(F_{2})$$ 
where $R_{0}$'s are the exceptional divisors introduced by $\gamma$ over $F_{0}$'s.
By Theorem \cite[I.Theorem 9,.1 (vii)]{BHPV} and the Riemann-Hurtwiz formula \cite[I.Section 16 equation (20)]{BHPV} we have:
$$ \gamma^\ast K_Y+2R_{0}
=K_{\widetilde{Y}} \qquad K_{\widetilde{Y}}= \pi^\ast K_{Z}+R_{0} + \gamma^{-1}(F_{2}). $$
Therefore we obtain:
\begin{equation}\label{canonical_formula}
  \pi^\ast K_{Z}=R_{0} - \gamma^{-1}(F_{2}).  
\end{equation}
If there are no codimension $1$ submanifolds in $\text{Fix}(\alpha_Y)$, \emph{i.e.} there are no $F_{2}$'s components, by the formula \eqref{canonical_formula} we have $\pi^\ast K_Z$ is effective and $K_Z$ too. Hence $k(Z)\ge 0$. Since the Kodaira dimension cannot increase under quotient and it is a birational invariant we get $k(Y/\langle \alpha_Y \rangle)=k(Z)\le k(Y)=0$ and so $k(Z)=0$. 
Otherwise, $\pi^\ast K_Z$ is not effective and so $K_Z$, hence $k(Y/\langle \alpha_Y \rangle)=k(Z)=-\infty$. 
\end{proof}
\section{The Calabi-Yau $3$-folds of type $A$ with the group $\fD_4$} \label{sec: familyd4}
In this section we recall the construction of the family of Calabi-Yau $3$-folds of type $A$ with the group $\fD_4$ presented in \cite{CD2}.\\

We consider the abelian $3$-fold  $A':= E\times E\times E'$ where
\begin{equation}\label{group H1}
\begin{aligned}
& E:=E_\tau=\bC/(\bZ\oplus \tau \bZ ) \quad \tau\in \mathfrak{h} &&E':=E_{\tau'}=\bC/(\bZ\oplus \tau' \bZ)  \quad \tau'\in \mathfrak{h} \\
\end{aligned}
\end{equation}
and $\mathfrak{h}=\{z\in \bC\mid \text{Im}(z)>0\}$ is the upper half plane.
We define:
\begin{equation}\label{group H2}
\begin{aligned}
& r(\ul{z}):=(z_2, -z_1, z_3+u_3) &&s(\ul{z}):=(z_2+u_1,z_1+u_2,-z_3) \\
& (u_1,u_2) \in (E\times E)[2] \setminus\{(0,0)\} \text{ s.t. $u_1\not=u_2$} &&u_3\in E'[4]\setminus\{0\}.
\end{aligned}
\end{equation}
We denote $H:=\langle r,s \rangle$ and its faithful representation $\rho: H \xrightarrow{\qquad} \text{GL}_3(\bC)$:
\begin{equation}\label{representation rho}
    \rho(r)=\begin{pmatrix} 0 & 1 & \\ -1 & 0 & \\ & & 1 \end{pmatrix} \qquad 
    \rho(s)=\begin{pmatrix} 0 & 1 & \\ 1 & 0 & \\ & & -1 \end{pmatrix}
\end{equation}
We observe that $\rho$ decomposes into two irreducible representations $\rho_2\oplus \rho_1$ with dim$\rho_j=j$ for $j=1,2$.\\
We are free to choose $u_1, u_2, u_3$ in several way and in the following definition we choose a way.
\begin{definition}\label{def1X}
Let $E, E', r, s $ as above with  $u_1:=\dfrac{\tau+1}{2}, u_2:=\dfrac{\tau}{2}, u_3=\dfrac{1}{4}$. We denote $X:=\dfrac{A'}{\langle r,s\rangle}$ and $w:=s^2$.
\end{definition}

\begin{lemma}\label{manifoldX}
The algebraic manifold $X$ is a Calabi-Yau $3$-fold of type $A$ with the group $\fD_4$ and its Hodge numbers are $(h^{1,1}(X),h^{2,1}(X))=(2,2)$.
\end{lemma}
\begin{proof}
We observe $X=A/G$ where $A:=\dfrac{A'}{w}$ is an abelian $3$-fold and $G:=\dfrac{H}{w}$ acts freely on it.
By an easy direct computation, see also \cite[Theorem 1.1]{CD}, we see that $G$ defines a free action of $\mathfrak{D}_4$ on $A$ which does not contain any translation, hence $X$ is a projective GHM with the group $\mathfrak{D}_4$. Since the group $G\subset \text{SL}(3,\bC)$ then $K_Y=0$. 
Since $H^{i,j}(X)=H^{i,j}(A)^G$, by using \eqref{representation rho} we obtain $h^{1,0}(X)=h^{2,0}(X)=0$ and $h^{1,1}(X)=h^{2,1}(X)=2$. Thus, in particular, $X$ is a Calabi-Yau $3$-folds.
\end{proof}

In this paper, we consider $X$ both as the quotient $\dfrac{A'}{H}$ and $\dfrac{A}{G}$. By abuse of notation we still denote by $\rho$ the representations of $G$ on $\bC^3$.\\
Using the fact that there exists a unique faithful representation of $\mathfrak{D}_4$ over an abelian $3$-fold, see \cite[Theorem 0.1]{OS} and \cite{CD2}, and by the result \cite[Theorem 1]{CC} which describes the Teichm\"uller space of a GHM, one leads to the following Theorem.

\begin{theorem}[\cite{CD2}, Theorem $0.1$]\label{family d4}
The family above of Generalized Hyperelliptic $3$-folds $X$ with group $\mathfrak{D}_4$ forms an irreducible and $2$-dimensional family of complex manifolds. The K\"ahler manifolds with the same fundamental group as $X$ yield an open subspace of the Teichm\"uller space of $X$ parametrized by the periods $\tau$ and $\tau'$ of the elliptic curves $E$ and $E'$.
\end{theorem}

\begin{remark}
By Theorem \ref{family d4} we deduce that all GHM $3$-folds with the group $\fD_4$ are Calabi-Yau $3$-folds of type $A$ with the group $\fD_4$.
\end{remark}

\begin{definition}
We denote by $\cF^A_{\fD_4}$ the $2$-dimensional family of Calabi-Yau $3$-folds of type $A$ with the group $\fD_4$.
\end{definition}

\begin{remark}\label{modulispace}
Theorem \ref{family d4} tell us that all manifolds $Y\in \cF^A_{\mathfrak{D}_4}$ are given as the quotient of an abelian $3$-fold $E_\mu\times E_\mu\times E_{\mu'}$ by a free action of a group $H'$ of order $16$ which contains a normal subgroup $G'$ isomorphic to $\fD_4$ and $G'$ does not contain any translation. In particular, the space of parameters of $\cF^A_{\mathfrak{D}_4}$ is given by $\cM_{1,1}\times \cM_{1,1}$, where $\cM_{1,1}=\mathfrak{h}/\text{SL}_2(Z)$ is the moduli space of elliptic curves, see \cite[Section 4]{Ap}.
Thus $Y\in \cF^A_{\mathfrak{D}_4}$ is associated to the pair $(\mu,\mu')\in \cM_{1,1}\times \cM_{1,1}$ and viceversa.
\end{remark}

\section{The Picard group of $X\in \cF^A_{\fD_4}$}\label{sec: neronseveri}

In this section we explicitly describe the Picard group of $X\in \cF^A_{\fD_4}$. It is worth recalling that abstractly the Picard group of a Generalized Hyperelliptic Manifolds is described in \cite[Theorem 2.5]{Cat23}. Although, in this situation we know that the GHMs in $\cF^A_{\fD_4}$ are Calabi-Yau threefolds and so we can use the characterization of their Picard group given in Lemma \ref{Picgroup}.\\

By Lemma \ref{fund_group} we known that $X=A'/H=\bC^3/\Gamma$ with $\Gamma\le \text{Iso}(\bC^3)$ crystallographic group and $\Gamma\simeq \pi_1(Y)$. Moreover it holds
\begin{center}
    \begin{tikzcd}
0\longrightarrow\pi_1(A')\simeq \Lambda\longrightarrow\pi_1(X)\simeq\Gamma \longrightarrow H \longrightarrow 0
    \end{tikzcd}
    \footnote{We just observe that in this case $\Lambda$ is not the maximal abelian and normal subgroup of finite index.}
\end{center}

We denote by $\ol{r}, \ol{s}\in \Gamma$ lifts of $r,s\in H$ to $\bC^3$, respectively and we write $\Lambda=\{ \lambda(\ul{z})=(z_1+t_1,z_2+t_2, z_3+t_3) \mid t_i\in \Lambda_i \}$ where $\Lambda_1=\Lambda_2=\tau\bZ\oplus \bZ\simeq \pi_1(E)$ and $\Lambda_3=\tau'\bZ\oplus \bZ\simeq \pi_1(E')$.
We define three elements in $\Lambda$: 
\begin{align*}
    &\lambda_{1}: (z_1,z_2,z_3)\xmapsto{\quad} (z_1+1, z_2, z_3) \\
& \lambda_{2}: (z_1,z_2,z_3)\xmapsto{\quad} (z_1, z_2+1, z_3) \\
& \lambda_{3}: (z_1,z_2,z_3)\xmapsto{\quad} (z_1+\tau, z_2+\tau, z_3)
\end{align*}
and we consider $\Sigma_1:=\langle \lambda_1,\lambda_2,\lambda_3, [\overline{r},\overline{s}]\rangle$ a specific subgroup of $\Gamma\simeq \pi_1(X)$.
    
\begin{theorem}\label{commutator group}
Let $X$ be as in Definition \ref{def1X}, we have the following isomorphism:
$$ \bA\text{b}(\pi_1(X)) \simeq (\bZ/4\bZ)\times (\bZ/4\bZ)\times (\bZ/2\bZ)=\langle \overline{r}, \overline{s}, \lambda_{\tau'} \rangle$$ where $\lambda_{\tau'}(\ul{z})=(z_1,z_2,z_3+\tau')$.
\end{theorem}

\begin{proof}
Let us consider $\lambda_{1,\tau}: (z_1, z_2, z_3)\mapsto (z_1+\tau, z_2, z_3)$ and
$2\Lambda$ in $\pi_1(X)$. The following relations hold:
    \begin{align*}
    & (z_1-2t_1,z_2-2t_2,z_3)=[\lambda,(\overline{r})^2](\ul{z}) \quad \forall \lambda\in \Lambda \\
    & (z_1,z_2,z_3-2t_3)=[\widetilde{\lambda},\overline{s}](\ul{z})  \quad \forall \widetilde{\lambda}(\ul{z}) =(z_1,z_2,z_3+t_3)\in \pi_1(A')\simeq \Lambda \\
& \lambda_1=[(\overline{r}), (\overline{s})^2] \text{(mod}2\Lambda)
\qquad \lambda_2=[\overline{r}^{-1}, (\overline{s})^2]  \text{(mod}2\Lambda) \qquad \lambda_3=[\lambda_{1,\tau},\overline{r}]  \text{(mod}2\Lambda). 
\end{align*}
Thus $\langle 2\Lambda, \Sigma_1 \rangle \le [\pi_1(X),\pi_1(X)]$ and since it is normal we get the following diagram:
\begin{equation}\label{ab}
    \begin{tikzcd}
    \pi_1(X) \arrow[r,"\varphi"] \arrow[dr, two heads] &
    \Sigma:=\dfrac{\pi_1(X)}{\langle 2\Lambda,\Sigma_1\rangle}\arrow[d,"f", two heads] \\ 
    & \bA\text{b}(\pi_1(X)).
    \end{tikzcd}
\end{equation}
We get that the generators of 
$ \Sigma$ are $\{\varphi(\overline{r}),\varphi(\overline{s}), \varphi(\lambda_{\tau'}),\varphi(\lambda_{1,\tau}),\varphi(\lambda_{2,\tau}) \}$
where $\varphi(\lambda_{\tau'})(\ul{z})=(z_1, z_2, z_3+\tau')$ and $\varphi(\lambda_{2,\tau})(\ul{z})=(z_1, z_2+\tau, z_3)$. 
According to the following relations
\begin{align*}
&\varphi(\lambda_{2,\tau})=\varphi(\lambda_{1,\tau})^{-1} && \varphi(\lambda_{1,\tau})=\varphi(\overline{r}\overline{s})^2=\varphi(\overline{r})^2\varphi(\overline{s})^2.
\end{align*}
we get that $\Sigma$ is generated by 
$\varphi(\overline{r}),\varphi(\overline{s}),\varphi(\lambda_{\tau'})$. It is easy to check that $\Sigma$ is an abelian group. From the fact that $\varphi(\lambda_{\tau'})$ has order two and $\varphi(\overline{r}),\varphi(\overline{s})$ have order four the only possibility is $\Sigma\simeq \bZ/4\bZ \times \bZ/4\bZ \times \bZ/2\bZ$. Since the commutator subgroup $[\pi_1(X),\pi_1(X)]$ is the smallest subgroup of $\pi_1(X)$ such that the quotient group is abelian, we obtain $[\pi_1(X),\pi_1(X)]=\langle 2 \Lambda, \Sigma_1\rangle $. Thus $f$ in diagram \eqref{ab} is an isomorphism and $\bA$b$(\pi_1(X))\simeq\Sigma$.
\end{proof}

\begin{corollary}\label{mainth3-1}
Let $X\in \cF^A_{\mathfrak{D}_4}$ be as in Definition \ref{def1X}. 
The Picard group is: $$\textup{Pic}(X)\simeq\bZ^2 \oplus (\bZ/4\bZ)\times (\bZ/4\bZ)\times (\bZ/2\bZ) $$
where
$ \bZ/4\bZ\times \bZ/4\bZ\times \bZ/2\bZ=\langle \overline{r}, \overline{s}, \lambda_{\tau'} \rangle$  is defined in Theorem \ref{commutator group}.
\end{corollary}
\begin{proof}
By combining \ref{Picgroup}, $h^{1,1}(X)=2$ and the description of $\bA$b$(\pi_1(X))$ in Theorem \ref{commutator group}, we obtain the result.
\end{proof}

\section{Automorphism group of manifolds in $\mathcal{F}_{\mathfrak{D}_4}^A$}\label{sec: automorphismsgroup}

\noindent In this section we are going to prove the following result:

\begin{theorem}\label{mainth1}
Let us assume that $\textup{End}_{\bQ}(E)\not\simeq \bQ(\zeta_6)$.
The automorphism group of $X$ is isomorphic to $\textup{Aut}(X)\simeq (\bZ/2\bZ)^4$ whose ele\-ments are induced by order two translations on $A'$ satisfying \begin{equation}\label{translation}
 t_1+t_2\in \{0,\frac{1}{2}\} \quad t_1\in E[2] \quad t_3\in E'[2]. 
 \end{equation} In particular, every $\alpha_{X}$ preserves the volume form of $X$.
\end{theorem}

\textbf{Notation:} We denote by $\text{Aut}_0(T)$ the group of invertible endomorphism of a complex torus $T$.

\begin{remark}\label{lifting under isogeny}
Let $T_1=\bC^n/\Lambda_1$ and $T_2=\bC^n/\Lambda_2$ be two $n$-dimensional complex tori, we consider an isogeny $f: T_1\xrightarrow{ m:1 } T_2$. Let $\alpha_2\in \text{Aut}(T_2)$, $\eta: \langle \alpha_2 \rangle \longrightarrow \text{GL}_n(\bC)$ be its representation. Then $\alpha_2$ admits $m$ lifts to $T_1$ if and only if $\eta(\alpha_2)\in \text{Aut}_0(T_1)$. 
\end{remark}

\begin{remark}
 We observe that $ A=B\times E' $ where $B=\dfrac{E\times E}{w_{|E\times E}}$. We denote by $G_2=G_{|B}$ and $G_1=G_{|E'}$: it is easy to see that $G_1\simeq G_2\simeq \fD_4$.   
\end{remark}

\begin{proposition}\label{Aut(X)}
Let us assume that $\textup{End}_{\bQ}(E)\not\simeq \bQ(\zeta_6)$. Then for $X\in \cF^A_{\fD_4}$ we have $\textup{Aut}(X)\simeq\dfrac{\textup{N}_{\textup{Aut}(E\times E)\times \textup{Aut}(E')}(H) }{H}$.
\end{proposition}
\begin{proof}
According to Corollary \ref{aut GHV} we have $\text{Aut}(X)\simeq \dfrac{\text{N}_{\text{Aut}(A)}(G)}{G}$. We consider the following diagram:
\begin{center}
\begin{tikzcd}
 \textup{N}_{\textup{Aut}(A')}(H) \arrow[d, "\theta_2"']  \arrow[dr,"\theta" ]&  \\
  \textup{N}_{\textup{Aut}(A)}(G) \arrow[r,twoheadrightarrow, "\theta_1"']  & \textup{Aut}(X) 
 \end{tikzcd}
\end{center}
We prove the surjectivity of $\theta$ by proving the one of
$\theta_2$ and we proceed as follow.
It is easy to observe that $ \textup{N}_{\textup{Aut}(A)}(G) = \textup{N}_{\textup{Aut}(B)}(G_2) \times  \textup{N}_{\textup{Aut}(E')}(G_1)$ and $\textup{N}_{\textup{Aut}(A')}(H)=\textup{N}_{\textup{Aut}(E\times E)}(\langle G_2,w_2 \rangle)\times \textup{N}_{\textup{Aut}(E')}(G_1)$. 
Thus, to prove the surjectivity of $\theta_2$, it's enough to prove that every $\alpha_B$ admits a lift on $E\times E$. 
According to \cite[Theorem 1.2 (1)]{Og}, $\text{Aut}(X)$ is finite and since $G$ is finite then  $\text{N}_{\text{Aut}(A)}(G)$ is finite too. 
Thus, we aim to prove that the maximal automorphism group of finite order on $B$ admits a lift on $E\times E$. Furthermore, since $B$ admits an action of $\fD_4$ (induced by the action of $\fD_4$ on $E\times E$) we restrict our attention to the maximal automorphism group $G'$ of finite order on $B$ which contains a subgroup isomorphic to $\fD_4$.
We observe that $B$ can not split into the product of two isomorphic elliptic curves. If it did, the element $g:=\begin{pmatrix} 0 & -1 \\ 1 & 1 \end{pmatrix}\in \text{Aut}_0(B)$ and in particular $g$ would amdit two lifts $\widetilde{g_\epsilon}=g(\ul{z})+\epsilon w_2$ for $\epsilon\in\{0,1\}$ which belong to $\text{N}_{\text{Aut}(E\times E)}(w_2)$. This condition leads to a contradiction.
By \cite[Tables $8$ and $9$]{Fuj}  it follows that
whenever $\textup{End}_\bQ(E)\not\simeq \bQ(\zeta_6)$, $G'=\langle \rho_2(r),\rho_2(s) \rangle\simeq \fD_4 \subseteq \text{Aut}_0(E\times E)$.
Hence, by Remark \ref{lifting under isogeny}, $\alpha_B$ admits a lift to $E\times E$.
\end{proof}

\begin{corollary} \label{AutB}
With the notation above, every $\alpha_{A'}\in \textup{N}_{\textup{Aut}(E\times E)\times \textup{Aut}(E')}(H)$ admits a representation $\eta=\eta_2\oplus \eta_1$ on $\bC^3$ such that $\eta_2(\alpha_{A'})=\rho_2(r^js^i)$ for $j=0,1,2,3$ and $i=0,1$.
\end{corollary}

\noindent \textbf{Assumption:} From now on we assume $\text{End}(E)\not\simeq\bQ(\zeta_6)$.

\begin{lemma}\label{teclemma} Let $\alpha_{A'}\in \textup{N}_{\text{Aut}(E\times E)\times \text{Aut}(E')}(H)$ such that $\alpha_{A'}(\ul{z})=\eta(\alpha_{A'})\ul{z}+t_{\alpha_{A'}}$ where $t_{\alpha_{A'}}$ is the translation by the point $=(t_1, t_2, t_3)$ and $\eta$ is a representation of $\alpha_{A'}$ on $\bC^3$. The following conditions hold :
\begin{center}
    $\eta(\alpha_{A'})=\rho(r^js^i)$ \quad for some $j=0,1,2,3$ \quad $i=1,2$ \\
    $t_1+t_2\in \{0,\frac{1}{2}\} \quad t_1\in E[2]  \quad 2t_3=\begin{cases} 0 & \text{ if }j=0,2 \\
    \dfrac{1}{2} & \text{ if } j=1,3 . \end{cases}$
\end{center}
\end{lemma}
\begin{proof} 
By computing the condition $\alpha_{A'}\in \text{N}_{\text{A'}}(H)$ and using Corollary \ref{AutB}, one find the relations in the statement.
\end{proof}

An easy check shows that every $\alpha_{A'}\in\textup{N}_{\text{Aut}(A')\
}(H)$ as the following characterization:
\begin{equation}\label{representation on A'}
\alpha_{A'}(z)=\rho(r^js^i)(z)+\begin{cases}t_r+t_{\alpha_{A'}} & \text{ if } j=1,3, \forall i \\  
t_{\alpha_{A'}} & \text{ if } j=0,2, \forall i \end{cases} \text{ with $t_{\alpha_{A'}}$ satisfying \eqref{translation}}. 
\end{equation}

\begin{proof}[Proof of Theorem \ref{mainth1}]
By Lemma \ref{teclemma} we have $\abs{\text{N}_{\text{Aut}(A)}(H)}=2^8$, hence $\abs{\text{Aut}(X)}=2^4$.
Let us consider the following translations in $\text{N}_{\text{Aut}(A)}(H)$:
\begin{equation}\label{trasn2}
\begin{cases}t_r+t_{\alpha_{A'}}-t_{r^js^i} & \text{ if }j=1,3, \forall i \\ t_{\alpha_{A'}}+ t_{r^js^i} & \text{ if }j=0,2, \forall i\end{cases}
\end{equation}
with $t_{\alpha_{A'}}$ satisfying $\eqref{translation}$.
An easy computation shows that the translations above and every $\alpha_{A'}$ in $\text{N}_{\text{Aut}(A')}(H)$ differ by an element in $H$, hence they induce the same automorphism on $X$. Therefore, every $\alpha_X\in \text{Aut}(X)$ is induced by a translation on $A'$ by a point $(t_1,t_2,t_3)$ satisfying \eqref{translation}. Thus $\text{Aut}(X)$ is abelian and so $\text{Aut}(X)\simeq (\bZ/2\bZ)^4$. Last statement follows since every $\alpha_X$ is induced by a translation $\alpha_{A'}$ and $\omega_X=(\pi_H)_\ast\bigl(\omega_{A'}\bigr)$.
\end{proof}
\section{The quotients of $X\in \cF^A_{\fD_4}$}\label{sec: quotient X}
\noindent In this section we classify all the possible quotients of $X\in \cF^A_{\fD_4}$.\\

We observe that for $\alpha_X\in \text{Aut}(X)$
\begin{equation}\label{fix-equation}
 \text{Fix}(\alpha_X)=\pi_H(\{(\ul{z})\in A'   \mid  \exists h\in H \text{ with } \alpha_{A'}(\ul{z})=h(\ul{z})\})   
\end{equation}
where $\pi_H: A'\longrightarrow X=\dfrac{A'}{H}$ and $\alpha_{A'}$ is a lift of $\alpha_X$ on $A'$.

\begin{theorem}\label{mainth2}
Let $X$ be as in Definition \ref{manifoldX} and $\Upsilon\le \textup{Aut}(X)$. Then, each quotient $X/\Upsilon$ admits a crepant resolution $\beta: Y \longrightarrow X/\Upsilon$ where $Y$ is a Calabi-Yau $3$-fold.
In particular, there exist exactly $2$ automorphisms $(\alpha_1)_X$ and $(\alpha_2)_X$ acting freely on $X$ induced respectively by 
$ \alpha_1(\ul{z}):=(z_1, z_2, z_3+\frac{\tau'}{2})$ and $\alpha_2(\ul{z}):=(z_1, z_2, z_3+\frac{\tau'}{2}+\frac{1}{2}) $. Moreover, the $\dfrac{X}{(\alpha_j)_X}$'s belong to $\cF^A_{\mathfrak{D}_4}$.

\end{theorem}

\begin{proof}
By Theorem \ref{mainth1}, $\Upsilon\simeq (\bZ/2\bZ)^m$ for some $1\le m\le 4$ and every $\upsilon$ in $\Upsilon$ preserves the volume form of $X$. Since $\Upsilon$ is abelian, we can split the quotient $X/\Upsilon$ in a subsequent quotients of order two. 
Let $\alpha_X\in \Upsilon$: according to Proposition \ref{quotientCYthreefold}, $X/\alpha_X$ is birational to a Calabi-Yau $3$-fold $\widetilde{X}$. 
$\Upsilon_1:=\Upsilon/\langle\alpha_X\rangle\simeq  (\bZ/2\bZ)^{m-1}$: since $\Upsilon$ is abelian and preserves $\omega_X$ then 
$\Upsilon_1$ lifts to an action on $\widetilde{X}$ which preserves the volume form $\omega_{\widetilde{X}}$. Thus, by iterating the argument above we conclude that $X/\Upsilon$ admits resolution of singularities $\beta: Y \rightarrow X/\Upsilon$ with $Y$ a Calabi-Yau $3$-fold and so $\beta$ is a crepant resolution.\\
Let $\alpha_{A'}(\ul{z})=(z_1+t_1, z_2+t_2, z_3+t_3)\in \text{N}_{\text{Aut}(A')}(H)$. It induces a free action on $X$ if and only if  for every $h\in H$ the equation $\alpha_{A'}(\ul{z})=h(\ul{z})$ has no solutions: by a direct computation this happens if and only if $t_1=t_2\in \{0,\frac{1}{2}\}$ and $t_3\in\{\frac{\tau'}{2},\frac{\tau'+1}{2}\}$.
Fixed $t_3$ we have 
$(z_1+\frac{1}{2}, z_2+\frac{1}{2}, z_3+t_3)$ and $(z_1, z_2, z_3+t_3)$ differ by $w\in H$,
hence they define the same automorphism on $X$. Therefore, there are only two auto\-morphisms which act freely on $X$ denoted by $(\alpha_j)_X$ as in the statement for $j=1,2$.
Let us denote by $Y_j:=X/(\alpha_j)_X$. As we prove above, they are Calabi-Yau $3$-folds. By construction $Y_j= \dfrac{A'}{\langle \alpha_{j},H \rangle}$.
It is easy to see that $Y_j$ can be also obtain as free quotients of the abelian $3$-fold $A_j=\dfrac{A'}{\langle \alpha_{j},w \rangle}$  by the free action of $\dfrac{\langle \alpha_j, H\rangle }{\langle \alpha_{j},w \rangle}$. This latter group  does not contain any translation and has cardinality $8$. Thus $Y_j$'s are Calabi-Yau $3$-folds of type $A$ with a group of order $8$. By the classification \cite[Theorem 0.1]{OS} the only possibility is $\langle r_{A_j},s_{A_j}\rangle \simeq \mathfrak{D}_4$. 
\end{proof}
 
We compute $\text{Fix}(\alpha_X)$ as varying $\alpha_X\in \text{Aut}(X)$.
We use \eqref{fix-equation} and we denote by $\alpha_{A'}(\ul{z})=(z_1+t_1, z_2+t_2, z_3+t_3)$ a lift of $\alpha_X$ on $A'$ as in Theorem \ref{mainth1}. 
We refer to \ref{sec: fixlocus} for the proof. 
Let us fix $t_1\in E[2]$, $t_1+t_2\in \{0,\frac{1}{2}\}$ and $t_3\in E'[2]$ and we define the elliptic curves:
\begin{equation}\label{elliptic-curves}
 \resizebox{0.6\hsize}{!}{
    $\begin{split}
        &C^{1}_{p,q}=\{ (p,q,l)\in A' \mid l\in E', 2p=t_1, 2q=t_2\}  \\
        &C^{2}_{p,q}=\{ (p,q,l)\in A' \mid l\in E', 2p=t_1+\frac{1}{2}, 2q=t_2+\frac{1}{2}\}\\
        &C^{3}_{q,t}=\{ (p,q,l)\in A' \mid p\in E, 2q=\frac{1}{2}, 2l=\frac{1}{4}+t_3\} \\
        &C^{4}_{p,t}=\{ (p,q,l)\in A' \mid q\in E, 2p=\frac{1}{2}, 2l=\frac{3}{4}+t_3\} \\
        &C^{5,t_1}_l=\{(p,-p+u_1+t_1,l)\in A' \mid p\in E, 2l=\frac{1}{2}+t_3\} \\
        &C^{6,t_1}_l=\{(p,-p+u_2+t_1,l)\in A' \mid p\in E, 2l=\frac{1}{2}+t_3\} \\
        &C^{7,t_1}_l=\{(p,p+u_1+t_1,l)\in A' \mid p\in E, 2l=t_3\} \\
        &C^{8,t_1}_l=\{(p,p+u_2+t_1,l)\in A' \mid p\in E, 2l=t_3\} \\
        &C^{9}_{q,l}=\{(p,q,l)\in A' \mid p\in E, 2q=0, 2l=\frac{1}{4}+t_3\} \\
        &C^{10}_{p,l}=\{(p,q,l)\in A' \mid q\in E, 2p=0, 2l=\frac{3}{4}+t_3\} 
    \end{split}$
    }
\end{equation}

\begin{table}[H]
      \scalebox{0.7}{
    \centering
\begin{tabular}{|c|c|c|c|c|}
\hline \rule[-4mm]{0mm}{1cm}
 $t_1$ & $t_2$ & $t_3$ & $\textup{Fix}(\alpha_X)$ & $\abs{\textup{Fix}(\alpha_X)}$  \\ \hline \rule[-4mm]{0mm}{1cm}
 $0$ & $0$ & $\frac{\tau'}{2}$ & $\emptyset$ & $0$\\
 \hline \rule[-4mm]{0mm}{1cm}
  $0$ & $0$ & $\frac{\tau'+1}{2}$ & $\emptyset$ & $0$ \\
 \hline \rule[-4mm]{0mm}{1cm}
$0$ & $0$ & $\frac{1}{2}$ & $\pi_H(C^1_{0,0})$,  $\pi_H(C^{1}_{\frac{\tau}{2},\frac{\tau}{2}})$,  $\pi_H(C^{1}_{0,\frac{\tau}{2}})$, $\pi_H(C^{2}_{\frac{1}{4},\frac{1}{4}})$, $\pi_H(C^{2}_{\frac{1}{4},\frac{1}{4}+\frac{\tau}{2}})$ & $5$ 
\\   \hline \rule[-4mm]{0mm}{1cm}
$\frac{\tau}{2}$ & $\frac{\tau}{2}$ & $\frac{1}{2}$ & 
$\pi_H(C^{1}_{\frac{\tau}{4},\frac{\tau}{4}})$, $\pi_H(C^{2}_{\frac{\tau+1}{4},\frac{\tau+1}{4}})$, $\pi_H(C^3_{\frac{1}{4},\frac{3}{8}})$, $\pi_H(C^3_{\frac{1}{4},\frac{3}{8}+\frac{\tau'}{2}})$ & $4$ 
\\ \hline \rule[-4mm]{0mm}{1cm}
$\frac{\tau}{2}$ & $\frac{\tau}{2}$ & $\not=\frac{1}{2}$ & $\pi_H(C^3_{\frac{1}{4},\beta})$, $\pi_H(C^3_{\frac{1}{4},\beta+\frac{\tau'}{2}})$  \small\text{ with $2\beta=\frac{1}{4}+t_3$}& $2$  \\
\hline \rule[-4mm]{0mm}{1cm}
$\frac{\tau}{2}$  & $\frac{\tau+1}{2}$ & $\frac{1}{2}$ & $\pi_H(C^{1}_{\frac{\tau}{4},\frac{\tau+1}{4}})$,  $\pi_H(C^{1}_{\frac{\tau}{4},\frac{3\tau+1}{4}})$,  $\pi_H(C^{9}_{0,\frac{3}{8}})$,  $\pi_H(C^{9}_{0,\frac{3}{8}+\frac{\tau'}{2}})$,  $\pi_H(C^{5,\frac{\tau}{2}}_{0})$, $\pi_H(C^{5,\frac{\tau}{2}}_{\frac{\tau'}{2}})$ & $6$ 
 \\  \hline \rule[-4mm]{0mm}{1cm}
 $\frac{\tau}{2}$  & $\frac{\tau+1}{2}$ & $\not=\frac{1}{2}$ &  $\pi_H(C^{9}_{0,\beta})$, $\pi_H(C^{9}_{0,\beta+\frac{\tau'}{2}})$, $\pi_H(C^{5,\frac{\tau}{2}}_{\gamma})$, $\pi_H(C^{5,\frac{\tau}{2}}_{\gamma+\frac{\tau'}{2}})$ \small\text{ with $2\beta=\frac{1}{4}+t_3$ and $2\gamma=\frac{1}{2}+t_3$} & $4$ 
  \\  \hline \rule[-4mm]{0mm}{1cm}
 
 $0$ & $\frac{1}{2}$ & $\frac{1}{2}$ &  $\pi_H(C^{1}_{0,\frac{1}{4}})$,   $\pi_H(C^{1}_{0,\frac{1}{4}+\frac{\tau}{2}})$,  $\pi_H(C^{5,0}_{0})$, $\pi_H(C^{5,0}_{\frac{\tau'}{2}})$ & $4$ 
 
 \\ \hline \rule[-4mm]{0mm}{1cm}
 $0$ & $\frac{1}{2}$ & $\not=\frac{1}{2}$ & $\pi_H(C^{5,0}_{\gamma})$, $\pi_H(C^{5,0}_{\gamma+\frac{\tau'}{2}})$  \small\text{ with $2\gamma=\frac{1}{2}+t_3$} & $2$ \\ \hline 
\end{tabular} }
\caption{The fixed locus of $\alpha_X$ on $X$}
\label{table2}
\end{table}

\section{Hodge numbers and fundamental group of quotients of $X\in \cF^A_{\fD_4}$}\label{section8}
We classify the Calabi-Yau $3$-folds $Y$ as in Theorem \ref{mainth2}.
We summarize our results in a series of tables: if $\Upsilon\simeq (\bZ/2\bZ)$ in Table \ref{table3},
if $\Upsilon\simeq (\bZ/2\bZ)^2$ in Table \ref{table10},
if $\Upsilon\simeq (\bZ/2\bZ)^3$ in Table \ref{table12},
if $\Upsilon\simeq (\bZ/2\bZ)^4$ in Table \ref{table14}.

\subsection{Hodge numbers of crepant resolution of $X/\Upsilon$}
We briefly summarize the techniques to compute Hodge number of $Y$ as in Theorem \ref{mainth2} via the orbifold cohomology (cf. \cite{CR}). \\

\noindent Let $M$ be a manifold and $\alpha_M$ be an involution on $M$.
The action of $\alpha_M$ can be locally linearized near $m\in \text{Fix}(\alpha_M)$ by $\text{diag}\bigl((-1)^{a_1(m)},\dots,(-1)^{a_n(m)}\bigr)$ with $a_i(m)\in {0,1}$.
We define \emph{the age of $\alpha_M$ at $m$} as $age(\alpha_M):=\dfrac{\sum\limits_{i=1}^n a_i(m)}{2}$.\\

We recall the formula of orbifold cohomology in the situation which is of our interest, see \cite[Definition 3.2.4]{CR}: $X$ is the Calabi-Yau $3$-fold constructed in Section \ref{sec: familyd4} and $\Upsilon\le \text{Aut(X)}$ is characterized by Theorem \ref{mainth1}. 
Since $\text{Aut(X)}$ is abelian, the set of representatives of the conjugacy classes of $\Upsilon$ is equal to $\Upsilon$ and the centralizer of each element is $\Upsilon$ itself. Let $F^i_{\alpha_X}$ be an irreducible component of $\text{Fix}(\alpha_X)$ for each $id_X\not=\alpha_X\in \Upsilon$. As we have observed in the proof of Proposition \ref{quotientCYthreefold}, the action of $\alpha_X$ can be locally linearized as $\text{diag}(1,-1,-1)$ 
hence the age of $\alpha_{X}$ in a fixed point is equal to $1$. The orbifold cohomology of $X/\Upsilon$ is:

\begin{equation}\label{orbifoldscohomology}
  H^{p,q}_{orb}(X/\Upsilon):=H^{p,q}(X)^\Upsilon\oplus \bigoplus_{id_X\not=\alpha_X\in \Upsilon} \bigoplus_{i} H^{p-1,q-1}(F^i_{\alpha_X}/\Upsilon). 
\end{equation}
We denote by $h^{p,q}_{orb} (X/\Upsilon)$ the dimension of $H^{p,q}_{orb}(X/\Upsilon)$.  We also recall the following important result:

\begin{theorem}[Theorem $5.4$ and  Corollary $6.15$, \cite{BD96}]\label{hodgeorbifold} Let $M$ be a compact, K\"ahler, complex manifold of dimension $n$ with trivial canonical bundle and equipped with an action of a finite group $G'$ such that preserves the volume form of $M$. Assume the existence of a crepant resolution $\beta: \widetilde{M}\longrightarrow M/G'$. Then $$ h^{p,q}(\widetilde{M})=h^{p,q}_{orb}(M/G').$$
\end{theorem}

\begin{proposition}\label{Hdg of quotient of X}
Let $X\in \cF^A_{\mathfrak{D}_4}$ and $\Upsilon\le\textup{Aut}(X)$. Let $Y\longrightarrow X/\Upsilon$ be the crepant resolution constructed in Theorem \ref{mainth2}, we have:
$$ h^{1,1}(Y)=h^{2,1}(Y)=2+\sum\limits_{id\not=\alpha_X\in \Upsilon}\abs{\dfrac{\textup{Fix}(\alpha_X)}{\Upsilon}}. $$ 
In particular $e(Y)=0$. 
\end{proposition}
\begin{proof}
If $X/\Upsilon$ is smooth, by \ref{mainth2} part (ii) we have  $Y=X/\Upsilon\in \cF^A_{\mathfrak{D}_4}$ so the result easily follows. Assume that $X/\Upsilon$ is singular and let $Y$ be as in the statement. We apply Theorem \ref{hodgeorbifold}. Since $\Upsilon$ is abelian then $\dfrac{\coprod\limits_{\alpha_X\in \Upsilon}\text{Fix}(\alpha_X)}{\Upsilon}=\coprod\limits_{\alpha_X\in \Upsilon}\dfrac{\text{Fix}(\alpha_X)}{\Upsilon}$. By Table \ref{table2} each $F^i_{\alpha_X}\in \text{Fix}(\alpha_X)$ is an elliptic curve.
If $\Upsilon$ acts on $F^i_{\alpha_X}$, since $\Upsilon$ is induced by a group of translation on $A'$,
it can either fixes $F^i_{\alpha_X}$ or acts by translation on it. If it does not preserve $F^i_{\alpha_X}$, we have identifications in the quotients $X/\Upsilon$. In any cases $F^i_{\alpha_X}/\Upsilon$ is an elliptic curve. 
Since every element of $\Upsilon$ is induced by a translation, we get $h^{1,1}(X)^\Upsilon=h^{1,1}(X)=h^{2,1}(X)=h^{2,1}(X)^\Upsilon$. Therefore, by applying \eqref{orbifoldscohomology} and Theorem \ref{hodgeorbifold} we obtain:
\begin{align*}
h^{1,1}(Y)=h_{orb}^{1,1}(X/\Upsilon)&=h^{1,1}(X)^\Upsilon+\sum\limits_{id\not=\alpha_X\in \Upsilon}\sum\limits_{i} h^{0,0}(F^i_{\alpha_X}/\Upsilon) \\
&=h^{2,1}(X)^\Upsilon+\sum\limits_{id\not=\alpha_X\in \Upsilon}\sum\limits_{i} h^{1,0}(F^i_{\alpha_X}/\Upsilon)=h_{orb}^{2,1}(X/\Upsilon)=h^{2,1}(Y).
\end{align*}
Consequently $h^{1,1}(Y)=h^{2,1}(Y)$ and $e(Y)=0$. Since $h^{0,0}(F^i_{\alpha_X}/\Upsilon)=1$  and $h^{1,1}(X)^{\Upsilon}=2$ we have the result.
\end{proof}

To compute $\abs{\dfrac{\text{Fix}(\alpha_X)}{\Upsilon}}$ one can use the description of $\text{Fix}(\alpha_X)$ in Table \ref{table2} and then need to study the the action of $\Upsilon$ on it.
 
\subsection{Fundamental group of a desingularization of $X/\Upsilon$}\label{sec: fundgroup}
Through well-known techniques, we recall how to compute the fundamental group of an orbifold space.
For a complete discussion we refer to \cite[Chapter 11]{Br}.

\begin{definition}
Let $\Gamma$ be a group acting on a topological space. The action is called \emph{discontinuous} if for every $y\in Y$, the stabilizer $\Gamma_y$ is finite and there exists an open neighborhood $V_y$ such that $\gamma V_y \cap V_y=\emptyset$ for every $\gamma \not\in \Gamma_y$.
\end{definition}

\begin{theorem}[\cite{Br} Propositions 11.2.3 and 11.5.2 (c)]\label{fund. group}
Let $V$ be a connected topological space and $\Gamma$ be a group acting on $V$. If $\Gamma$ defines a discontinuous action on $V$ and $Y$ is simply connected, the fundamental group of $V/\Gamma$ is: 
$$\pi_1\biggl(\dfrac{V}{\Gamma}\biggr)\simeq\dfrac{\Gamma}{F_{\Gamma}}$$ where $F_{\Gamma}=\{ \gamma \in \Gamma \mid \exists \text{ z }\in V \text{ s.t. } \gamma\in \Gamma_z \}\unlhd \Gamma$.
\end{theorem}

\begin{lemma}\label{discoaction}
Every $\Gamma\le \textup{Iso}(\bR^n)=O(n)\ltimes \bR^n$ crystallographic group defines a discontinuous action on $\bR^n$.
\end{lemma}
\begin{proof}
Since $\bR^n$ is an Hausdorff topological space, for every $\ul{z}\in \bR^n$ there exists $V_{\ul{z}}$ such that $\gamma V_{\ul{z}} \cap V_{\ul{z}}=\emptyset$ for every $\gamma\in \Gamma$ which is not in the stabilizer $\Gamma_{\ul{z}}$ of $\ul{z}$. 
We prove that the stabilizer of $\ul{z}$ is finite. 
According to Remark \ref{exact-sequence}, $\Gamma$ fits in:
\begin{equation}
     0\longrightarrow \Lambda' \xrightarrow{\quad i\quad} \Gamma \xrightarrow{\quad l\quad} G' \longrightarrow 0
\end{equation}
where $\Lambda=\Gamma \cap \bR^n$ and $G'$ is finite.
We have:
$$ \Gamma_{\ul{z}}=\{ (M,\lambda')\in\Gamma \mid l(M,\lambda)=M\in G', i^{-1}(M,\lambda')=\lambda'\in \Lambda' \text{ such that } \lambda'(\ul{z})=M^{-1}(\ul{z})\}. $$ 
We see that $\lambda'$ is uniquely determined by $M$ and $M$ varies in the finite group $G'$, we deduce that $\Gamma_{\ul{z}}$ is finite.
\end{proof}

In our situation: we have $X/\Upsilon$ with $X\in \cF^A_{\fD_4}$ and $\Upsilon\le \text{Aut}(X)$. According to Theorem \ref{aut GHV}, $X/\Upsilon$ can be obtained as finite quotient of $A$ and so there exist $\Gamma \le \text{Iso}(\bC^3)$ such that $X/\Upsilon=\bC^3/\Gamma$. The group $\Gamma$ is a crystallographic group since it can be viewed as finite extension of $\pi_1(A)$ by the group $\langle \Upsilon_{A}, G\rangle $. 
Thus $\Gamma$ defines a discontinuous action on $\bC^3$ and we can apply Lemma \ref{fund. group} to compute $\pi_1(X/\Upsilon)$.

\begin{theorem}[Theorem 7.8 \cite{Kollar}]
Let $Y_1$ be a normal analytic space and let us consider a resolution of singularities $f: Y_2\longrightarrow Y_1$. If $Y_1$ has quotient singularities then $\pi_1(Y_2)$ and $\pi_1(Y_1)$ are isomorphic.
\end{theorem}

\noindent Consequently we obtain:

\begin{corollary}\label{fund-quotient}
Let $X\in \cF^A_{\mathfrak{D}_4}$ and $Y$ as in Theorem \ref{mainth2}. Then: $$\pi_1(Y)\simeq \pi_1(X/\Upsilon )\simeq\dfrac{\Gamma}{F_{\Gamma}}$$ where $Y=\bC^3/\Gamma$ as explained before.
\end{corollary}

\begin{remark}\label{computation-fund-group}  
We briefly explain how to compute $\pi_1(X/\Upsilon)$. Assume $\Upsilon=\langle\alpha_X\rangle$ and denote by $\alpha_{A'}$ a lift of $\alpha_X$ on $A'$.
We have $X/\Upsilon$ is also the quotient $\bC^3/\Gamma_\alpha$ with $\Gamma_\alpha\le \text{Iso}(\bC^3)$ where $\Gamma_\alpha$ is the finite extension of $\Lambda=\pi_1(A')$ by the group $\langle H,\alpha_{A'}\rangle$. We denote by 
$\ol{H}$ and $\ol{\alpha_{A'}}$ lifts of $H$ and $\alpha_{A'}$ to $\bC^3$, respectively. 
To compute $F_{\Gamma_{\alpha_{A'}}}$, one needs to study $\ol{\alpha_{A'}}(\ul{z})=(\lambda\circ \ol{h})(\ul{z})$ for every $\lambda\in \Lambda$, $\ol{h}\in \ol{H}$ and $\ul{z}\in \bC^3$: to do this one can use
 Remark \ref{remark: fixed locus alphaX}. This remark can be used more generally to compute $F_{\Gamma_\Upsilon}$ for every $\Upsilon \le \text{Aut}(X)$ and one can easily observe that
 as $\Upsilon$ grows then $\pi_1(Y)$ tends to the identity group, \emph{i.e.} $Y$ is simply-connected.

\end{remark}

\noindent We summarize all the computations in the tables in Appendix \ref{appendix}. In particular, we deduce the following result. 

\begin{corollary}
Let $X\in \cF^A_{\fD_4}$ and $Y$ be the Calabi-Yau $3$-fold obtained as desingularization of $X/\Upsilon$ as varying $\Upsilon \in \text{Aut}(X)$.
There are $19$ different pairs $(h^{1,1}(Y), \pi_1(Y))$ and so there are at least $19$ non homeomorphic Calabi-Yau $3$-folds.    
\end{corollary}

\noindent From the description of $\pi_1(Y)$ we deduce the topology of its universal cover. 

\begin{corollary}\label{fundgroup}
Let $X$ as in Definition \ref{manifoldX} and $Y$ be a desingularization of $X/\Upsilon$ with $\Upsilon$ in $\text{Aut}(X)$ and $\pi_1(Y)$ is described in the tables in Appendix \ref{appendix}. The following cases appear.
\begin{romanenumerate}
\item If $\pi_1(Y)$ is infinite and it is a rank-$6$ lattice, $Y$ is a Calabi-Yau $3$-fold of type $A$ which belongs to $\cF^A_{\fD_4}$.
\item If $\pi_1(Y)$ is infinite and it is a finite extension of a rank-$2$ lattice, the universal cover of $Y$ is isomorphic to $\bC\times S$ where $S$ is a K$3$ surface.
\item If $\pi_1(Y)$ is finite and not trivial, its universal cover is a simply-connected Calabi-Yau $3$-fold.
\end{romanenumerate}
\end{corollary}
\begin{proof}
Let $\widetilde{Y}$ be the universal cover of $Y$. By the Beauville-Bogomolov decomposition theorem \cite[Theorem 14.15]{GHJ}:
$$\widetilde{Y}= \bC^n \times \prod\limits_{i} W_i \times \prod\limits_{j}Z_j$$
where $W_i$'s are simply connected Calabi-Yau manifolds and $Z_j$'s are irreducible holomorphic symplectic manifolds. If $\pi_1(Y)$ is a rank-$6$ lattice then $\widetilde{Y}=\bC^3$ and we have the first statement. If $\pi_1(Y)$ is a finite extension of a rank-$2$ lattice the only possibility is $\widetilde{Y}=\bC\times W_1$ where $W_1$ is a K$3$ surface. In the last case we cannot find a complex space as factor of $\widetilde{Y}$ and since irreducible holomorphic symplectic manifolds exist only in even dimension, $\widetilde{Y}$ is a simply-connected Calabi-Yau $3$-fold.
\end{proof}

Let us consider case (ii) of Corollary \ref{fundgroup}: these manifolds, in fact, admit the product of a K$3$ surface and an elliptic curve as étale Galois cover.  Indeed $\pi_1(Y)$ fits in the following exact sequence
\begin{equation}\label{exactfundgroup}
0\longrightarrow L_Y \longrightarrow \pi_1(Y) \longrightarrow G' \longrightarrow 0
\end{equation}
where $L_Y$ is a rank-$2$ lattice and $G'$ is a finite group; so $Y$ admits a $G'$-cover isomorphic to the product of a K$3$ surface and the elliptic curve $\bC/L_Y$.

\begin{definition}
A \emph{Calabi-Yau $3$-fold of type K} is a Calabi-Yau $3$-fold which admits the product of a K$3$ surface and an elliptic curve as finite étale Galois cover.
\end{definition}
 
\begin{remark} Similar to Remark \ref{notranslation}, we assume that the Galois group $G'$ of the cover a Calabi-Yau $3$-fold of type $K$ does not contains any element of type $(id,t)$ where $t$ is a translation on the elliptic curve. Thus, we say a \emph{Calabi-Yau $3$-fold of type $K$ with the group $G'$}.
\end{remark}

By looking at the tables in Appendix \ref{appendix}, case (ii) of Corollary \ref{fundgroup} appears only for two groups $\Upsilon \le \text{Aut}(X)$ and in both cases the desingularization $Y$'s of $X/\Upsilon$ are Calabi-Yau $3$-folds of type $K$ with the group $(\bZ/2\bZ)^2$. More precisely,
 
\begin{corollary}
Let $Y$ be a Calabi-Yau $3$-fold of type $K$ with the group $(\bZ/2\bZ)^2$ obtained as desingularization of $X/\Upsilon$ with $X\in\cF^A_{\fD_4}$ and $\Upsilon \le \textup{Aut}(X)$. The followings hold.
\begin{romanenumerate}
\item If $\Upsilon=\langle \alpha_X \rangle$ is induced by $\alpha_{A}(\ul{z})=(z_1, z_2, z_3+\dfrac{1}{2})\in \textup{Aut}(A)$, the $(\bZ/2\bZ)^2$-étale cover is isomorphic to Km$_2(B)\times E''$ where Km$_2(B)$ is the Kummer surface and $E''=E'/\alpha_{A}$.

\item If $\Upsilon=\langle \alpha_X , \beta_X \rangle$ is induced by $\alpha_{A}(\ul{z})=(z_1, z_2, z_3+\dfrac{1}{2})$ and $\beta_A(\ul{z})=(z_1, z_2, z_3+\dfrac{\tau'}{2})$ in $\textup{Aut}(A)$, the $(\bZ/2\bZ)^2$-étale cover is isomorphic to Km$_2(B)\times E'$.

\end{romanenumerate}
\end{corollary}

\begin{proof}
\begin{romanenumerate}
\item Let us consider the group $\Gamma=\langle r,s,\alpha_A \rangle \le\text{Aut}(A)$. 
We obtain the following diagram:
\begin{center}
\begin{tikzcd}
A \arrow[d, "16:1" '] \arrow[r,"2:1"] & A/\alpha_A = B\times E''\arrow[d,"2:1"] \\
A/\Gamma  & (B\times E'')/r^2 \arrow[l,"4:1","g"']  \\
Y \arrow[u,"\beta"]&  Km_2(B)\times E'' \arrow[u,"\gamma"] \arrow[l, "4:1", "f"'] 
\end{tikzcd}
\end{center}
where $\beta$ and $\gamma$ are resolution of singularities.
Since $\langle \alpha_A \rangle \unlhd \Gamma$, $(\alpha_A)_{|B}=id$ and $(\alpha_{A})_{|E'}=(r_1)^2$ is a translation we obtain $\dfrac{A}{\alpha_A}= B \times \dfrac{E'}{(r_1)^2}$ and we denote the elliptic curve $E'':=\dfrac{E'}{(r_1)^2}$. The automorphism $r^2$ descends to an automorphism on $B\times E''$ which acts as $\text{diag}(-1,-1)$ on the abelian surface $B$ and as the identity on $E''$, hence $(B\times E'')/r^2$ is birational to $Km_2(B)\times E''$.
The action of the group $\Gamma/\langle \alpha_A, r^2 \rangle$  defines free action of $(\bZ/2\bZ)^2$ on $(B\times E'')/r^2$. Moreover it preserves the singular locus of $(B\times E'')/r^2$, thus it lifts to a free action of $(\bZ/2\bZ)^2$ on $Km_2(B)\times E''$ which does not contain any element $(id, t)$ with $t$ a translation on $E''$. Therefore, the quotient $f$ is the $(\bZ/2\bZ)^2$-étale cover of $Y$ whose Galois group is induced by $\langle r, s \rangle$.


\item Let us consider the group $\Gamma=\langle r, s, \alpha_A,\beta_A \rangle \le\text{Aut}(A)$. 
Similar to the previous case we obtain the following diagram:
\begin{center}
\begin{tikzcd}
A \arrow[d, "32:1" '] \arrow[r,"4:1"] & A/\langle \alpha_A,\beta_A \rangle \simeq B\times E'\arrow[d,"2:1"] \\
A/\Gamma  & (B\times E')/r^2 \arrow[l,"4:1","g"']  \\
Y \arrow[u,"\beta"]&  Km_2(B)\times E' \arrow[u,"\gamma"] \arrow[l, "4:1", "f"'] 
\end{tikzcd}
\end{center}
where we observe that $\langle \alpha_A,\beta_A\rangle \simeq E'[2]$, hence $A/\langle \alpha_A,\beta_A\rangle\simeq A=B\times E'$.
As above, the free action of the group  $\langle r, s \rangle$ induces a free action of $(\bZ/2\bZ)^2$ on $Km_2(B)\times E'$ which does not contain any element $(id, t)$ with $t$ a translation on $E'$. Therefore, the quotient $f$ is the $(\bZ/2\bZ)^2$-étale cover of $Y$ whose Galois group is induced by $\langle r, s \rangle$.
\end{romanenumerate}

\end{proof}

\section{A map of degree two on $\cF^A_{\fD_4}$}\label{sec: map on family} 
We construct a degree two map in $\cF^A_{\fD_4}$ by using the existence of quotients of $X$ which belong again to $\cF^A_{\fD_4}$.

\begin{lemma}\label{doublecover}
Let $X$ as defined in Definition \ref{manifoldX}. There are two double covers of $X$ which belong to $\cF^A_{\fD_4}$ and they are isomorphic. \end{lemma}
\begin{proof}
According to Theorem \ref{mainth2}, there exist exactly $2$ quotients $X/(\alpha_j)_X$ of $X$ which belong again to $\cF^A_{\fD_4}$. Equivalently, $X$ is the double étale quotient of other two manifolds $Y_1, Y_2 \in \cF^A_{\fD_4}$ which fit the following commutative diagram
\begin{center}
\begin{tikzcd}[column sep=large]
A'_1 \arrow[dr] \arrow[d,"16:1"']& A' \arrow[d]& A'_2 \arrow[dl] \arrow[d,"16:1"]\\
Y_1 \arrow[r, "2:1" '] & X & Y_2 \arrow[l, "2:1"] 
\end{tikzcd}
\end{center}
where $A'_j=E_{\nu_j}\times E_{\nu_j}\times E_{\nu'_j}$, with $\nu_j, \nu'_j\in \mathfrak{h}$, are the $H_j$-étale cover of $Y_j$ by Remark \ref{modulispace} with $\abs{H_j}=16$. Moreover, by Theorem \ref{mainth2}  we have $X=Y_1/(\beta_1)_{Y_1}$ and $X=Y_2/(\beta_2)_{Y_2}$ where $\beta_j$'s are translations on $A'_j$ respectively by the points $(0,0, \frac{\nu'_1}{2})$ and $(0,0, \frac{\nu'_2+1}{2})$.
Since $A'=A'_j/\beta_j$ for $j=1,2$ we have $$\tau=\nu_1=\nu_2 \qquad 2\tau'=\nu'_1 \qquad 2\tau'=\nu'_2+1. $$
In particular, we see that $\nu'_2=\mathfrak{T}(\nu'_1)$ with $\mathfrak{T}=\begin{pmatrix} 1 & 1 \\ 0 & 1
\end{pmatrix}\in \text{SL}_2(\bZ)$. Thus $\nu'_2$ and $\nu'_1$ define the same point in $\cM_{1,1}$ and so $E'_{\nu'_1}\simeq E_{\nu'_2}$. Consequently, $Y_1\simeq Y_2$. 
\end{proof}

\begin{remark} We denote by $Y_{\mu,\mu'}$ the Calabi-Yau $3$-fold in $\cF^A_{\fD_4}$ whose $16$-étale cover is $E_\mu\times E_\mu \times E_{\mu'}$. \end{remark}

\noindent Combining Lemma \ref{doublecover} and Theorem \ref{mainth2}, we obtain the following result.
\begin{theorem}\label{mainth4}
There exists a $2:1$ map $f: \quad \cF^A_{\fD_4} \xrightarrow{\quad 2:1\quad} \cF^A_{\fD_4}$ such that the image of $Y_{\mu,\mu'}\in \cF^A_{\fD_4}$ is its étale double cover $Y_{\mu, 2\mu'}$ in $\cF^A_{\fD_4}$ and the preimages of $Y_{\mu, \mu' }\in \cF^A_{\mathfrak{D}_4}$ are its two étale quotients $Y_{\mu, \frac{\mu'}{2}}$ and $Y_{\mu, \frac{\mu'}{2}+\frac{1}{2}}$ in $\cF^A_{\fD_4}$.
\end{theorem}

\begin{remark}
It is worth noting that the existence of exactly two quotients of $X$ which belong again to $\cF^A_{\fD_4}$ tell us that there exist others two constructions for an abelian $3$-folds with a free action of $\fD_4$: they are $(E_\mu\times E_\mu \times E_{\mu'})/\text{T}_j$ where $\text{T}_j=\langle w, \gamma_j\rangle$ and $\gamma_j$ are the translation by the point $(0,0,\frac{\mu'+j}{2})$ for $j=0,1$. This result was already stated in \cite[Theorem 2.7]{HK}, but we observe that  there is an error on the description of $\gamma_j$.
\end{remark}

\section{The Calabi-Yau $3$-folds of type $A$ with the group $(\bZ/2\bZ)^2$}\label{section 10}

\subsection{Action of finite group on complex tori}
We refer to \cite{CC}.

\begin{definition}
Let $\Lambda$ be a free abelian group of even rank and $G\longrightarrow \text{GL}(\Lambda)$ be a faithful representation of a finite group $G$. A \emph{$G$-Hodge decomposition} is a decomposition into $G$-invariant linear subspaces:
$$\Lambda\otimes \bC=H^{1,0} \oplus H^{0,1} \quad \ol{H^{1,0}}=H^{0,1}.$$
\end{definition}

\noindent Splitting $\Lambda\otimes \bC$ into isotypic components, we write $\Lambda \otimes \bC=\bigoplus\limits_{\chi\in \text{Irr}(G)} U_{\chi}$ and $U_\chi =W_\chi \otimes M_\chi$. Here, $W_\chi$ is the $\bZ$-module corresponding to the irreducible representation $\chi$ and $M_\chi\simeq \bC^{m_\chi}$ encodes how many times the representation with character $\chi$ appears in the decomposition. \\
Thus, we obtain:
$$V :=H^{1,0} =\bigoplus\limits_{\chi\in \text{Irr}(G)}V_\chi$$
with $V_\chi=W_\chi \otimes M_\chi^{1,0}$.

\begin{definition}
The \emph{Hodge type of a $G$-Hodge decomposition} is the collection of the dimensions $\nu(\chi)=\text{dim}_{\bC}M^{1,0}_\chi$. Here, $\chi$ runs over all non-real characters.
\end{definition}

\begin{remark}\label{Hdg decomposition}
\begin{romanenumerate}
\item For $\chi$ non real it holds $\nu(\chi)+\nu(\ol{\chi})=\text{dim}_{\bC}(M_\chi)$.
\item All $G$-Hodge decompositions of a fixed Hodge type are parametrized
as follows: for a real irreducible
character $\chi$, one chooses a $\dfrac{1}{2}$dim$_{\bC}(M_\chi$)-dimensional subspace
of $M_\chi$, and for a non-real irreducible character, one chooses a $\nu(\chi)$-dimensional subspace of $M_\chi$ . 
\item The Hodge type is a invariant by deformation, see \cite[Theorem 81]{Cat15}.
\end{romanenumerate}
\end{remark}

Therefore, to give a classification of the complex tori $T=\bC/\Lambda$ with a free action of $G$ (which does not contain any translation), one needs to determine all possible complex structure on $T$ such that the action of $G$ is holomorphic. This corresponds to determine all possible $G$-Hodge decompositions on $\Lambda$.

\subsection{The construction of the family $\cF^A_{(\bZ/2\bZ )^2}$}\label{family 2Z}
Let $A/G$ be a Calabi-Yau $3$-fold of type $A$ with $G\simeq (\bZ/2\bZ)^2$, by \cite[Theorem 0.1]{OS}
the group $G=\langle a, b \rangle$ has an unique faithful representation $\rho: G \longrightarrow \text{GL}_3(\bC)$:
\begin{equation}\label{representation rho'}
    \rho(a)=\text{diag}(-1,-1,1) \qquad 
    \rho(b)=\text{diag}(-1,1,-1)
\end{equation}
We observe that $\rho=\rho_1\oplus\rho_2\oplus\rho_3$ with $\rho_j$ one dimensional irreducible representations. 
We consider the abelian $3$-fold $A:=E_1\times E_2\times E_3$ where $E_j:= \bC/(\bZ\oplus \tau_j\bZ)$ is an elliptic curve with $\tau_j\in \mathfrak{h}$, for $j=1,2,3$ and the automorphisms:
$$ a(\ul{z})=(-z_1,-z_2,z_3+u_3) \qquad b(\ul{z})=(-z_1+u_1, z_2+u_2, -z_3) $$
where $u_j\in E_j[2]\setminus\{0\}$. One can easily prove that $G=\langle a, b \rangle$ defines a free action of $(\bZ/2\bZ)^2$ on $A$. 

\begin{definition}\label{CY Z2}
With the notation above and fixed $u_j\in E_j[2]\setminus\{0\}$ for $j=1,2,3$, we define the Calabi-Yau $3$-fold $X:=A/G$ of type $A$ with the group $(\bZ/2\bZ)^2$.
\end{definition}

\noindent By using \eqref{representation rho'} we find: 
\begin{equation}\label{Hdg of X}
 H^{1,1}(X)\simeq \langle dz_j \wedge d\ol{z_j} \rangle_{j=1,2,3} \qquad H^{2,1}(X)\simeq \langle dz_j \wedge dz_k\wedge d\ol{z_j} \rangle_{i\not=k\not=j=1,2,3}
 \end{equation}
hence $h^{1,1}(X)=h^{2,1}(X)=3$, see also \cite[Theorem 0.1]{OS}.

\begin{theorem}\label{mainth6}
There exists a $3$-dimensional family of $(\bZ/2\bZ)^2$-equivariant complex structures on a $3$-dimensional complex torus $T$. 
The family of Calabi-Yau $3$-folds of type $A$ with the group $(\bZ/2\bZ)^2$ is irreducible and $3$-dimensional.
\end{theorem}
\begin{proof}
Let $T=V/\Lambda$ with $V=\bC^3$ be a complex torus.
By \cite[Theorem 0.1]{OS} 
we have a unique $(\bZ/2\bZ)^2$-decomposition on $\bC^3$ given by $V=V_{\chi_1}\oplus V_{\chi_2}\oplus V_{\chi_3}$ where $\chi_j$ are the irreducible characters corresponding to $\rho_j$. Hence:
$\Lambda\otimes \bC=V\oplus \ol{V}=\bigoplus\limits_{j=1}^3 V_{\chi_j}\otimes M_{\chi_j}$
where $V_{\ol{\chi_j}}=V_{\chi_j}$ since $\chi_j$ are real characters and dim$_{\bC}M_{\chi_j}=2$. 
By Remark \ref{Hdg decomposition} the parameters of the $(\bZ/2\bZ)^2$-Hodge decomposition are given by the choice of a $1$ dimensional subspace $M^{1,0}_{\chi_j}$ of $M_{\chi_j}$ for $j=1,2,3$. Hence we have a $3$-dimensional family of $(\bZ/2\bZ)^2$-equivariant complex structures on $T$.
Let $X$ be a Calabi-Yau $3$-fold of type $A$ with the group $(\bZ/2\bZ)^2$: by Theorem \ref{deform_CY3}, the space that parametrizes its local deformation is $H^{2,1}(X)$ which is irreducible and has dimension $3$, see \eqref{Hdg of X}. We conclude that the family of Calabi-Yau $3$-folds of type $A$ with the group $(\bZ/2\bZ)^2$ is irreducible and $3$-dimensional.
\end{proof}

\begin{remark}
We also deduce that the space which parametrizes the Calabi-Yau $3$-folds of type $A$ with the group $(\bZ/2\bZ)^2$ is $(\cM_{1,1})^3$. \end{remark}

\begin{definition}
We denote the family constructed in Theorem \ref{mainth6} by $\cF^A_{(\bZ/2\bZ)^2}$.
\end{definition}

\subsection{The Picard group of $X\in \cF^A_{(\bZ/2\bZ)^2}$}

\begin{theorem}\label{mainth6}
Let $X\in \cF^A_{(\bZ/2\bZ )^2}$ as in Definition \ref{CY Z2}. 
The Picard group is 
$$\textup{Pic}(X)\simeq\bZ^3 \oplus \pi_1(X). $$ 
\end{theorem}
\begin{proof}
It follows by \eqref{Picgroup}, $h^{1,1}(X)=3$ and the fact that $\pi_1(X)$ is abelian. 
\end{proof}

\subsection{Automorphism group and quotients of $X$}
\begin{theorem}\label{mainth7}
The automorphism group $\textup{Aut}(X)\simeq (\bZ/2\bZ)^7$ and its elements are induced by automorphisms on $A$ whose linear part is in $\langle \textup{diag}(-1,1,1)\rangle$ and the translation part is any translation of order two. 
\end{theorem}
\begin{proof}
By Corollary \ref{aut GHV} we have $\text{Aut}(X)\simeq \dfrac{\text{N}_{\text{Aut}(A)}(G)}{G}$. Let $\alpha_{A}\in \text{Aut}(A)$: we write $\alpha_{A}(\ul{z})=\eta(\alpha_A)(z)+t_{\alpha_A}$ where $\eta$ is its representation and $t_{\alpha_A}$ is the translation by the point $(t_1,t_2,t_3)$. The condition $\alpha_A\in \text{N}_{\text{Aut}(A)}(G)$ leads to:
 \begin{center}
     $\eta=\eta_1\oplus \eta_2 \oplus \eta_3$ $\quad$ $\eta_j(\alpha_A)\in \langle -1\rangle$ $\quad$ $t_j\in E_j[2]. $
 \end{center}
We obtain that $\textup{Aut}(X)\simeq (\bZ/2\bZ)^7$. It is easy to check that every $\alpha_{A}\in \text{N}_{\text{Aut}(A)}(G)$ such that $\text{det}[\eta(\alpha_A)]=1$ but $\eta(\alpha_A)\not=id$ differs from any translation of order two by an element in $G$. 
Let $\alpha_{1},\alpha_2,\alpha_3, \alpha_4\in \text{N}_{\text{Aut}(A)}(G)$ as follows:
\begin{align*}
	& \alpha_1(\ul{z})=\text{diag}(-1,1,1) +t
	&& \alpha_2(\ul{z})=\text{diag}(1,-1,1) +t_a+t\\
	& \alpha_3(\ul{z})=\text{diag}(1,1,-1) +t_b+t 
	&& \alpha_4(\ul{z})=\text{diag}(-1,-1,-1) +t_{ab}+t 
\end{align*}
with $t$ any translation of order two on $A$ and $t_a$,$t_b$, $t_{ab}$ are the translation parts of $a, b, ab$ respectively.
One can check that $\alpha_2$, $\alpha_3$ and $\alpha_4$ differ from $\alpha_1$ by an element in $G$. 
Therefore we find the conditions as in the statement.
\end{proof}

\begin{remark}\label{fixed-locus-alpha_X}
Let $\alpha_X\in \textup{Aut}(X)$ be induced by $\alpha_A\in \text{Aut}(A)$.
\begin{itemize} 
    \item Assume that $\alpha_A$ is a translation by the point $(t_1,t_2,t_3)\in A$ with $t_i\in E_i[2]$. An easy computation shows that the fixed locus is a finite set of elliptic curves
    given by $\bigcup\limits _{id\not=g\in G}\pi(\text{Fix}(\alpha_Ag))$. In particular,
    \begin{align*}
        & \pi(\text{Fix}(\alpha_Aa))=\pi\biggl( \{(p,q,l)\in A \mid 2p=t_1, 2q=t_2, l\in E_3\}\biggr)\subset\text{Fix}(\alpha_X) \Leftrightarrow t_3=u_3\\
        &\pi(\text{Fix}(\alpha_Ab))=\pi\biggl( \{(p,q,l)\in A \mid 2p=t_1+u_1, 2l=t_3, q\in E_2\}\biggr)\subset\text{Fix}(\alpha_X) \Leftrightarrow t_2=u_2\\
        &\pi(\text{Fix}(\alpha_Aab))=\pi\biggl( \{(p,q,l)\in A \mid 2q=t_2+u_2, 2l=t_3+u_3, p\in E_1\}\biggr)\subset\text{Fix}(\alpha_X) \Leftrightarrow t_1=u_1.
     \end{align*}

    \item  Assume that $\alpha_A(\ul{z})=(-z_1+t_1,z_2+t_2,z_3+t_3)$ with $t_i\in E_i[2]$.
        By explicit computations we have the followings.
        The $0$-dimensional subset in $\textup{Fx}(\alpha_X)$ is given by $\pi(\text{Fix}(\alpha_Aab)$ and it is never empty. The $2$-dimensional set in $\text{Fix}(\alpha_X)$ is given by 
        \begin{align*}
            & \pi(\text{Fix}(\alpha_A))\not=\emptyset \Leftrightarrow t_2=0, t_3=0\\
            & \pi(\text{Fix}(\alpha_Aa))\not=\emptyset \Leftrightarrow t_1=0, t_3=u_3\\
            & \pi(\text{Fix}(\alpha_Ab))\not=\emptyset \Leftrightarrow t_1=u_1, t_2=u_2
        \end{align*}   
        and it consists of a finite numbers of abelian surfaces.
\end{itemize}
\end{remark}

\begin{theorem}\label{mainth8}
Let $X\in \cF^A_{(\bZ/2\bZ)^2}$ and $\Upsilon\le\text{Aut}(X)$. Let $\beta: Y\rightarrow X/\Upsilon$ be the blow up of the singular locus of $X/\Upsilon$. The followings hold.
\begin{romanenumerate}

\item If $\Upsilon$ preserves the volume form of $X$, $\beta$ is a crepant resolution and $Y$ is a Calabi-Yau $3$-fold. In particular, there are exactly $3^3-1$ automorphism $(\alpha_j)_X$ which acts freely on $X$. They are induced by the translations $\alpha_j\in \textup{Aut}(A)$ by the point $(t_1, t_2, t_3)$ such that $u_i\not=t_i \in E_i[2]$ and $X/(\alpha_j)_X$ belong to $\cF^A_{(\bZ/2\bZ)^2}$.

\item If $\Upsilon$ does not  preserve the volume form of $X$, we have the following cases.
\begin{romanenumerate}
\item [(a)] If there exists at least one $\alpha_X \in \Upsilon$ such that fixes surfaces on $X$ then $Y$ has negative Kodaira dimension.
\item  [(b)] Otherwise, $Y$ has trivial Kodaira dimension and its canonical bundle is not trivial.
\end{romanenumerate}

\end{romanenumerate}
\end{theorem}
\begin{proof}
\begin{romanenumerate}

\item  According to Theorem \ref{mainth7} $\Upsilon\simeq (\bZ/2\bZ)^m$ with $1\le m \le 6$: using similar argument to the one in the proof of Theorem \ref{mainth2}, we conclude that $\beta : Y \longrightarrow X/\Upsilon$ is a crepant resolution and $Y$ is a Calabi-Yau $3$-fold.
By Proposition \ref{quotientCYthreefold} and \ref{kodairadimension} since $\Upsilon$ preserves the volume form of $X$, it can act freely on $X$. One can easily show that the $\alpha_X \in \text{Aut}(X)$ acting freely on $X$ are the one induced by $\alpha_A$ translation on $A$ by the point $(t_1, t_2, t_3)$ with $t_i\not= u_i$. It is easy to see that Calabi-Yau $3$-fold $Y=A/\alpha_X$ can be obtain as $\dfrac{(A/\langle\alpha_A\rangle)}{(\langle \alpha_A, G\rangle/\langle\alpha_A\rangle)}$: since $(A/\langle\alpha_A\rangle)$ is an abelian $3$-fold with an action of $(\langle \alpha_A, G\rangle/\langle\alpha_A\rangle)\simeq (\bZ/2\bZ)^2$ and does not contain any translation, we have $Y\in \cF^A_{(\bZ/2\bZ)^2}$.

\item In this case there exists at least one element in $\Upsilon$ which does not preserve the volume form of $X$. Moreover whenever we consider two elements in $\Upsilon$ which does not preserve the volume form of $X$, their composition defines an element which preserves the volume form of $X$. Therefore we can split $\Upsilon$ in the direct product of two groups $\Upsilon_1\times \Upsilon_2$ where $\Upsilon_1\simeq \bZ/2\bZ=\langle \alpha_X \rangle$ whose generator does not preserve the volume form of $X$ and $\Upsilon_2 \simeq (\bZ/2\bZ)^{k-1}$ which preserves the volume form of $X$ where $\left \vert \Upsilon \right\rvert=2^k$ for some $k=1,\ldots,7$.
\begin{enumerate}
    \item[a.] Since there exist at least one elements $\alpha_X$ that fixes surfaces (codimension $1$ submanifolds), according to Proposition \ref{kodairadimension} part (ii) $X/\langle \alpha_X \rangle $ has negative Kodaira dimension. Since the Kodaira dimension cannot increase under quotients and is a birational invariant then $X/\Upsilon$ admits a desingularization $Y$ such that $k(Y)=-\infty$ and $h^{j,0}(Y)=0$.

\item[b.] Since $\Upsilon$ does not fix surfaces we have that $\Upsilon_1$ fixes only isolated points. Moreover by Proposition \ref{quotientCYthreefold} $\Upsilon_2$ either acts freely or fixes curves.
We first consider the quotient by $\Upsilon_2$ which, up to a desingularization, produces a Calabi-Yau $3$-fold $Z$, by part (i). The action of $ \alpha_X $ lifts to an action of $\alpha_Z$ on $Z$, since $\Upsilon$ is abelian.
Moreover, $\alpha_Z$ does not preserve $\omega_Z$ since $ \alpha_X$ does not preserve $\omega_X$. By Proposition \ref{kodairadimension}, $\alpha_Z$ fixes points or divisors or both of them on $Z$. We prove that $\alpha_Z$ fixes only points. Since the resolution $\gamma\colon Z \rightarrow X/\Upsilon_2$ is an isomorphism outside the blown up locus and $\alpha_X$ fixes only points on $X$, the only divisors that (perhaps) $\alpha_Z$ could fix are the exceptional divisors. 
If $\alpha_Z$ fixes an exceptional divisor introduced by $\gamma$ on $Z$, which is a $\bP^1$-bundle over a curve $C$, then $\alpha_X$ should fix a curve $C$ on $X$: this can happen if and only if $C$ is contained in surfaces fixed by $\alpha_X$, indeed $\alpha_X$ does not fix curves on $X$ by Proposition \ref{kodairadimension}. By hypothesis $\alpha_X$ does not fixed surfaces on $X$, therefore $\alpha_Z$ fixes only points on $Z$. Thus we have a Calabi-Yau threefolds $Z$ with an involution that does not preserves its volume forms and fixes only points, by Proposition \ref{kodairadimension} we have $Z/\alpha_Z$ admits a desingularization $Y$ such that $k(Y)=0$ but not trivial canonical bundle. 
\end{enumerate}

\end{romanenumerate}
\end{proof}

\subsection{Hodge numbers of a desingularizations of quotients of $X\in \cF^A_{(\bZ/2\bZ)^2}$}
In this section we compute the Hodge numbers of $Y$ as in Theorem \ref{mainth8}.

\begin{proposition}\label{hdgnumbers}
Let $X\in \cF^A_{(\bZ/2\bZ)^2}$ and $\Upsilon\le\textup{Aut}(X)$. 
\begin{romanenumerate}
\item If $\Upsilon$ preserves the volume form of $X$, then there exist $\beta\colon Y\longrightarrow X/\Upsilon$ a crepant resolution such that
$$ h^{1,1}(Y)=h^{2,1}(Y)=3+\sum\limits_{id\not=\alpha_X\in \Upsilon}\left \lvert{\dfrac{\textup{Fix}(\alpha_X)}{\Upsilon}}\right\rvert. $$ 
In particular $e(Y)=0$.
\item If $\Upsilon=\Upsilon_1\times \Upsilon_2$ (as in proof of Theorem \ref{mainth8}) does not preserve the volume form of $X$, there exist a resolution of singularities $\beta\colon Y\longrightarrow X/\Upsilon$
such that:
\begin{align*}
     &h^{1,1}(Y) = 3+\sum\limits_{id\not=\upsilon\in \Upsilon_2}\left\lvert{\dfrac{\textup{Fix}(\upsilon)}{\Upsilon}}\right\rvert+16\\
     &h^{2,1}(Y) =
         \left \lvert\{F^i_\upsilon/\Upsilon \text{ elliptic curve }\}_{id\not=\upsilon\in \Upsilon_2} \right \rvert +16 
\end{align*}
where $F^i_\upsilon$ is a curved fixed by $\upsilon\in \Upsilon_2$ on $X$.
Furthermore, $\beta$ is the blowing up of $\textup{Sing}(X/\Upsilon)$ which introduces exactly one irreducible divisor on each irreducible submanifold blown up in $X/\Upsilon$.
\end{romanenumerate}
\end{proposition}

\begin{proof}
\begin{romanenumerate}
\item If $\Upsilon$ preserves the volume form of $X$, by Theorem \ref{mainth8} $\beta$ is a crepant resolution and so we use the orbifold cohomology formula \eqref{orbifoldscohomology} to compute the Hodge numbers of $Y$. 
We have $\text{Fix}(\Upsilon)=\coprod\limits_{\upsilon\in \Upsilon_A}\coprod\limits_{g\in G} \pi(\text{Fix}(\upsilon g))$ where $ \Upsilon_A$ is a lift of $\Upsilon$ to $A$ and $\pi: A\longrightarrow X$. According to Remark \ref{fixed-locus-alpha_X}, $\text{Fix}(\upsilon g)$ consists of a finite number of elliptic curves. Moreover, $G$ acts on each 
$\text{Fix}(\upsilon g)$ and since $G$ acts freely on $A$ it can either identify the elliptic curves in $\text{Fix}(\upsilon g)$ or act on them as a translation. In any case, $\text{Fix}(\Upsilon)$ consists of a finite numbers of elliptic curves.
We recall that each $\upsilon \in \Upsilon$ is induced by a (order two) translation on $A$.
Applying the same proof of Proposition \ref{Hdg of quotient of X} we obtain:
$$ h^{1,1}(Y)=h^{2,1}(Y)=3+\sum\limits_{id\not=\alpha_X\in \Upsilon}\left\lvert{\dfrac{\textup{Fix}(\alpha_X)}{\Upsilon}}\right\rvert. $$ 
In particular $e(Y)=0$.

\item We have that $\Upsilon$ does no preserve the volume form of $X$. Following the proof of Theorem \ref{mainth8}: we can write $\Upsilon$ as the direct product of the groups $\Upsilon_1\times \Upsilon_2$ where $\Upsilon_1=\langle \alpha_X \rangle$ is cyclic of order $2$ which does not preserve the volume form of $X$ and $\Upsilon_2 \simeq (\bZ/2\bZ)^{k-1}$ preserves the volume form of $X$ where $\left\lvert{\Upsilon}\right\rvert=2^k$ for $k=1,\dots,7$. We remark that the fixed locus of 
$\Upsilon_2$, if not empty, consists of a finite number of elliptic curves and the one of
$\Upsilon_1$ of finite numbers of isolated points and (possibly) smooth surfaces, see Remark \ref{fixed-locus-alpha_X}.
We consider the following commutative diagram:
\begin{equation}\label{diagram}
\begin{tikzcd}
& & A \arrow[d, "\pi" '] \\
& & X \arrow[d, "f" '] \\
{\overline{\Upsilon_1}}\Circlearrowright W \arrow[r, "\delta"] \arrow[d, "2:1" ', "q" ] & {\widetilde{\Upsilon_1}}\Circlearrowright Z \arrow[r,"\gamma"] & X/\Upsilon_2 \Circlearrowleft \Upsilon/\Upsilon_2  \arrow[d," 2:1", "p" '] \\
Y=W/ \overline{\Upsilon_1} \arrow[rr, "\beta" ']& & X/\Upsilon 
\end{tikzcd}
\end{equation}
Here $\gamma\colon Z \longrightarrow X/\Upsilon_2$ is the blow up of each singular curves in $X/\Upsilon_2$ and $Z$ a Calabi-Yau $3$-fold, see Theorem \ref{mainth8} part (i). Since $\Upsilon$ is abelian 
then $\Upsilon_1$ preserves the fixed locus of $\Upsilon_2$. Therefore, the action of $\Upsilon_1\simeq \Upsilon/\Upsilon_2$ on $X/\Upsilon_2$ 
extends to an action of $\widetilde{\Upsilon_1}$ on $Z$. According to Remark \ref{fixed-locus-alpha_X}, the fixed locus of $\widetilde{\Upsilon_1}$ consists of isolated fixed points and (possibly) smooth surfaces: $\delta$ blow ups the $0$-dimensional subset $\Pi$ of $\text{Fix}(\widetilde{\Upsilon_1})$. In particular, $\widetilde{\Upsilon_1}$ lifts to an action  of $\ol{\Upsilon_1}$ on $W$. 
Finally, $\beta$ is a composition of birational maps making the diagram commutative.\\
We compute the Hodge numbers of $Y$ by consi\-dering the morphisms $q$ and $(\gamma\circ \delta)$. We have $H^{i,j}(Y)\simeq H^{i,j}(W)^{\overline{\Upsilon_1}}$ and by using the formula \cite[Theorem 7.31]{V} to compute $H^{i,j}(W)$, we lead to the following:
\begin{equation}\label{com Y}
\begin{aligned} 
H^{i,j}(Y)&\simeq H^{i,j}(W)^{\overline{\Upsilon_1}}\simeq [H^{i,j}(Z) \oplus H^{0,0}(\Pi)]^{\widetilde{\Upsilon_1}}= H^{i,j}(Z) ^{\widetilde{\Upsilon_1}} \oplus  H^{0,0}(\Pi) .
\end{aligned}
\end{equation}
where the last equality follows since $\widetilde{\Upsilon_1}$ fixes $\Pi$.
Since $\gamma$ is a crepant resolution we use the formula \eqref{orbifoldscohomology} to describe the cohomology of $Z$:
\begin{equation}\label{com Z}
\begin{split}
H^{i,j}(Z)& \simeq H^{i,j}(X)^{\Upsilon_2} \oplus \bigoplus\limits_{id\not=\upsilon\in \Upsilon_2} \bigoplus\limits_{i}H^{i-1,j-1}(F^i_\upsilon/\Upsilon_2)\\
&=H^{i,j}(X) \oplus \bigoplus\limits_{id\not=\upsilon\in \Upsilon_2} \bigoplus\limits_{i}H^{i-1,j-1}(F^i_\upsilon/\Upsilon_2)
\end{split}
\end{equation}
where $F^i_\upsilon$ is a fixed elliptic curve by $\upsilon\in \Upsilon_2$ on $X$ and the last equality follows since $\Upsilon_2$ preserves the cohomology of $X$. By substituting \eqref{com Z} in \eqref{com Y} we obtain:
\begin{equation}\label{final_hdg_Y}
H^{i,j}(Y) \simeq  [H^{i,j}(X) \oplus \bigoplus\limits_{id\not=\upsilon\in \Upsilon_2} \bigoplus\limits_{i}H^{i-1,j-1}(F^i_\upsilon/\Upsilon_2)] ^{\Upsilon_1} \oplus  H^{0,0}(\Pi) .
\end{equation}
Using \eqref{Hdg of X} and the expression of the generator of $\Upsilon_1$ we obtain the followings: $H^{1,1}(X)^{\Upsilon_1}\simeq H^{1,1}(X)$ and $H^{2,1}(X)^{\Upsilon_1}=0$. 
If $\Upsilon_1\simeq \Upsilon/\Upsilon_2$ acts on an elliptic curve $F^i_\upsilon/\Upsilon_2$, the following situations situations can appear: it fixes the curve, it acts by translation on it or it is the hyperelliptic involution on it.
In the first two cases, $F^i_\upsilon/\Upsilon$ is an elliptic curve. In the second case, $F^i_\upsilon/\Upsilon$ is a rational curve.
If $\Upsilon_1\simeq \Upsilon/\Upsilon_2$ does not act on $F^i_\upsilon/\Upsilon_2$, we have identifications in the quotient $X/\Upsilon$ and $F^i_\upsilon/\Upsilon$ is still an elliptic curve. Therefore, by applying to \eqref{final_hdg_Y} we obtain:
\begin{align*}
     &h^{1,1}(Y) = 3+ \sum\limits_{id\not=\upsilon\in \Upsilon_2}\left\lvert{\dfrac{\textup{Fix}(\upsilon)}{\Upsilon}}\right\rvert+\left \lvert{\Pi}\right\rvert\\
     &h^{2,1}(Y) =
         \left \lvert\{F^i_\upsilon/\Upsilon \text{ elliptic curve }\}_{id\not=\upsilon\in \Upsilon_2} \right \rvert +\left \lvert{\Pi}\right\rvert
\end{align*}

We compute $\left \lvert{\Pi}\right\rvert$. We prove that $\Pi$ is isomorphic, via $\gamma$, to the $0$-dimensional subset fixed by $\Upsilon_1$ on $X/\Upsilon_2$.
This is equivalent
to prove that $\widetilde{\Upsilon_1}$ does not fix isolated points on the exceptional divisors introduced by $\gamma$. Since the exceptional divisors introduced by $\gamma$ on $Z$ are $\bP^1$-bundles over the curves in $\textup{Sing}(X/\Upsilon_2)$, we have to check that whenever $\widetilde{\Upsilon_1}$ acts on these $\bP^1$-bundles it does not fix isolated points on the fibers of these bundles.
We assume that $\widetilde{\Upsilon_1}$ acts on the $\bP^1$-bundles  introduced by $\gamma$. Let $E$ be one of these $\bP^1$-bundle over a curve $C=f(C')$ in $\text{Sing}(X/\Upsilon_2)$ which is the mage of a curve fixed $C'$ by $\upsilon\in \Upsilon_2$. Since $\widetilde{\Upsilon_1}$ acts on $E$, then $\alpha_X$ must acts on $C'$. We recall that $C'$ is an elliptic curve, hence $\alpha_X$ can act either as a translation or as the hyperelliptic involution. In the first case we obtain that  $\widetilde{\Upsilon_1}$ acts on $E$ but without fixing points on its fibers.
In the second case, $\alpha_X$ fixes four points on $C'$: we prove that these are not isolated fixed points by $\alpha_X$ on $X$ but that they lie on surfaces fixed by $\alpha_X\upsilon$ on $X$.
Let $X\simeq \bC^3_{(z_1,z_2,z_3)}$ and we assume that $C'$ over a point $P=(z_1,z_2)\in X$, \emph{i.e.} it has free variable on $z_3$: $\alpha_X$ acts on $C'$ if and only if $\alpha_X$ acts on $P=(z_1,z_2)\in X$ in the same way of $\upsilon$; therefore $\upsilon\alpha_X$ acts as the identity on $P=(z_1,z_2)$ and so it fixes surfaces on $X$. In particular, the four points on $C'$ fixed by $\alpha_X$ are the intersection points between the surfaces fixed by $\alpha_X\upsilon$ and $C'$. Thus, $\widetilde{\Upsilon_1}$ acts on $E$ by fixing four fibers.
This prove that $\Pi$ is isomorphic, via $\gamma$, to the $0$-dimensional subset fixed by $\Upsilon_1$ on $X/\Upsilon_2$.
By using Remark \ref{fixed-locus-alpha_X} we lead to the following description:
 \begin{equation}\label{eq-pi-points}
      \Pi=(\pi\circ f)\bigl(\bigsqcup_{\upsilon\in \Upsilon_2} I_\upsilon) \qquad I_\upsilon=\{p\in A\mid \alpha_A\upsilon ab(p)=p\}.
 \end{equation}    
 In particular, $\left \lvert{I_\upsilon}\right\rvert$ consists of $2^6$ points for every $\upsilon\in \Upsilon_2$. Since $G$ acts freely on $A$, then $G$ permutes the points in each $I_\upsilon$ and we get \begin{equation}
     \left \lvert{\pi(\bigsqcup_{\upsilon\in \Upsilon_2} I_\upsilon)}\right\rvert =\left \lvert{\Upsilon_2}\right\rvert\dfrac{2^6}{\left \lvert{G}\right\rvert}=\left \lvert{\Upsilon_2}\right\rvert2^4.
 \end{equation}
Since $\Upsilon$ is abelian then $\Upsilon_2$ acts on each $I_\upsilon/G$. In particular each $\upsilon_2\in \Upsilon_2$ permutes the points in $I_\upsilon/G$. Indeed, otherwise we would find isolated points fixed by $\alpha_X$ on a curve fixed by $\Upsilon_2$ and this situation, as explained above, cannot happen. Therefore, we lead to
 \begin{equation}
      \left \lvert\Pi\right\rvert=\dfrac{\left \lvert{\Upsilon_2}\right\rvert2^4}{\left \lvert{\Upsilon_2}\right\rvert}=2^4.
 \end{equation}

We also observe that the validity of \eqref{eq-pi-points} implies that the diagram \eqref{diagram} commutes and so $\beta$ is the blowing up of $\textup{Sing}(X/\Upsilon)$ which introduces exactly one irreducible divisor on each irreducible submanifold blown up in $X/\Upsilon$.

\end{romanenumerate}
\end{proof}

\begin{remark}
Let $X\in \cF^A_{(\bZ/2\bZ)^2}$ and $\Upsilon\le \text{Aut}(X)$. We can also compute the fundamental group of $Y$ desingularization of $X/\Upsilon$ by applying the results of Section \ref{sec: fundgroup}. In particular, in \cite[Section 1]{DW} one can find the fundamental groups and Hodge numbers of the Calabi-Yau $3$-folds $Y$ desingularizations of $X/\Upsilon$.
\end{remark}

\appendix
\section{Proof of Table \ref{table2}}\label{sec: fixlocus}

\begin{remark}\label{remark: fixed locus alphaX}
Let $\alpha(\ul{z})=(z_1+t_1, z_2+t_2, z_3+t_3)$ be a translation on $\bC^3$. Let us fix $\epsilon\in\{0,1\}$. We denote by $\ol{H}$ a lift of $H$ on $\bC^3$, we consider different $\ol{h}\in \ol{H}$ and study the equations $\alpha(\ul{z})=\ol{h}(\ul{z})$ on $\bC^3$.
In the followings, $u_1$ and $u_2$ are defined in Definition \ref{def1X}.
\begin{enumerate}

\item If $\ol{h}\in\{id, \ol{s}^2\}$: $\alpha(\ul{z})\in\{\ul{z},\ol{s}^2(\ul{z})\} \Leftrightarrow \alpha\equiv id \lor \alpha \equiv s^2 \Leftrightarrow \alpha_X=id_X.$

\item If $\ol{h}\in\{\ol{s}^{2\epsilon}\ol{r}^k\}$ for $k=1,3$:  $\alpha(\ul{z})=\ol{s}^{2\epsilon}\ol{r}^k(\ul{z})$  admit solution $\Leftrightarrow$ $t_3=\frac{1}{4}+\frac{\epsilon}{2}$.

\item If $\ol{h}\in\{\ol{s}^{2\epsilon}\ol{r}^2\}$ 
$\alpha(\ul{z})=\ol{s}^{2\epsilon}\ol{r}^2(\ul{z})$ admit solution $\Leftrightarrow$ $t_3=\dfrac{1}{2}$.

\item If $\ol{h}\in \{\ol{s}^{2\epsilon+1}\}$: $\alpha(\ul{z})\in\{\ol{s}(\ul{z}),\ol{s}^3(\ul{z})\} \Leftrightarrow t_1+t_2=\dfrac{1}{2}+\epsilon.$

\item If $\ol{h}\in \{\ol{r}^2\ol{s}^{2\epsilon+1}\}$: $\alpha(\ul{z})\in \{\ol{r}^2\ol{s}(\ul{z}),\ol{r}^2\ol{s}^3(\ul{z})\} \Leftrightarrow t_1-t_2=\dfrac{1}{2}.$

\item If $\ol{h}\in \{\ol{r}^3\ol{s}^{2\epsilon+1}\}$: $\alpha(\ul{z})\in \{\ol{r}^3\ol{s}(\ul{z}),\ol{r}^3\ol{s}^3(\ul{z})\}\Leftrightarrow t_2=u_1+\frac{\epsilon}{2}. $\item If $\ol{h}\in \{\ol{rs}^{2\epsilon+1}\}$: $\alpha(\ul{z})\in \{\ol{rs}(\ul{z}), \ol{rs}^3(\ul{z})\}  \Leftrightarrow t_1=u_2+\frac{\epsilon}{2}. $\end{enumerate}
\end{remark}

\begin{proof}[Proof of Table \ref{table2}]
Let $\alpha_X\in \text{Aut}(X)$ with $X\in \cF^A_{\fD_4}$ and $\alpha_{A'}$ its lift on $A'$ which a translation by $(t_1,t_2,t_3)$ as in Theorem \ref{mainth1}. 
According to \eqref{fix-equation},
we need to study the solutions of the equations $\alpha_{A'}(\ul{z})=h(\ul{z})$ with $h\in H$.
As evidence, we explicitly show one case.
Let us consider $t_1=t_2$ and $t_2=\frac{1}{2}$. By 
Remark \ref{remark: fixed locus alphaX} we need to consider only $\ol{h}\in\{\ol{r}^2,\ol{s}^2\ol{r}^2\}$: by explicit computations we find $\text{Fix}(\alpha_{A'})=I^1(t_1,t_2)\coprod I^2(t_1,t_2)$ where $I^j(t_1,t_2)=\coprod \limits_{2p=t_1,2q=t_2} C^j_{p,q}$. Then, we need to study the action of $H$ on the fixed locus. To do this we need to choose $t_1$. Let $t_1=0$ and $C^1_{p,q}\in \text{Fix}(\alpha_{A'})$. Since $-p=p$ and $-q=q$ we have:
\begin{center}
    \begin{tikzcd}
    r\Circlearrowright C^1_{p,p} \arrow[r,"s"] & C^1_{p+\frac{\tau+1}{2},p+\frac{\tau}{2}} \arrow[r,"s"] &C^1_{p+\frac{1}{2},p+\frac{1}{2}} \arrow[r,"s"] &C^1_{p+\frac{\tau}{2},p+\frac{\tau+1}{2}}
    \end{tikzcd} 
    \end{center}
    \begin{center}
    \begin{tikzcd}
   C^1_{p,p+\frac{\tau}{2}} \arrow[r,"s"] \arrow[d,leftrightarrow, "r"] & C^1_{p+\frac{1}{2},p+\frac{\tau}{2}} \arrow[r,"s"]  &C^1_{p+\frac{1}{2},p+\frac{\tau+1}{2}} \arrow[r,"s"]\arrow[d,leftrightarrow,"r"]  & C^1_{p,p+\frac{\tau+1}{2}}  \\
   C^1_{p+\frac{\tau}{2},p}  \arrow[r,"s"] & C^1_{p+\frac{\tau+1}{2},p} \arrow[r,"s"] &C^1_{p+\frac{\tau+1}{2},p+\frac{1}{2}}  \arrow[r,"s"]  & C^1_{p+\frac{\tau}{2},p+\frac{1}{2}} .
    \end{tikzcd}
\end{center} 
Thus for $t_1=0$: $\pi_H(I^1(t_1,t_2))=\{ \pi_H(C^1_{0,0}), \pi_H(C^1_{\frac{\tau}{2},\frac{\tau}{2}}), \pi_H(C^1_{0,\frac{\tau}{2}})\}$.
Similar computations for $C^2_{p,q}$ lead to
$\pi_H(I^2(t_1, t_2))=\{ \pi_H(C^2_{\frac{1}{4},\frac{1}{4}}), \pi_H(C^2_{\frac{1}{4},\frac{1}{4}+\frac{\tau}{2}})\}$. 
\end{proof}

\section{Tabulation of results for quotients of $X\in \cF^A_{\fD_4}$}\label{appendix}
\noindent Let $X\in \cF^A_{\fD_4}$ and $\Upsilon\le \text{Aut}(X)$.  We denote by $\alpha_X$ and $\beta_X$ and $\gamma_X$ the generators of $\Upsilon$ (depending on the cardinality of $\Upsilon$) and, as usual, by $\alpha_{A'}$ and $\beta_{A'}$ and $\gamma_{A'}$ respectively their lifts on $A'$.
We define the following sub-lattices of $\Lambda_3=\bZ\oplus \tau' \bZ$:
$\Lambda_3':=\bZ \oplus \dfrac{\tau'}{2}\bZ $, $\Lambda_3'':=\bZ \oplus \dfrac{\tau'+1}{2}\bZ$ and $\Lambda_3'''=\dfrac{1}{2}\bZ\oplus \bZ$.
We recall  the notation $\Lambda_1=\Lambda_2=\bZ\oplus \tau \bZ$.
\begin{table}[H]
    \centering
          \scalebox{0.78}{
\begin{tabular}{|c|c|c|}
\hline \rule[-4mm]{0mm}{1cm}
$\alpha_{A'}$ & $h^{1,1}(Y)=h^{2,1}(Y)$ & $\pi_1(Y)$ \\
  \hline \rule[-4mm]{0mm}{1cm}
 $(0,0,\frac{\tau'}{2})$ & $2$ & $0\rightarrow \Lambda_1\oplus \Lambda_2\oplus\Lambda_3' \rightarrow \pi_1(Y)  \rightarrow \fD_4 \rightarrow 0$ \\
  \hline \rule[-4mm]{0mm}{1cm}
 $(0,0,\frac{\tau'+1}{2})$ & $2$ & $0\rightarrow \Lambda_1\oplus \Lambda_2\oplus\Lambda_3'' \rightarrow \pi_1(Y)  \rightarrow \fD_4 \rightarrow 0$ \\
  \hline \rule[-4mm]{0mm}{1cm}
$(0,0,\frac{1}{2})$ & $7$&  $0\rightarrow \Lambda_3'''  \rightarrow \pi_1(Y)  \rightarrow  \bZ/2\bZ\times \bZ/2\bZ \rightarrow 0 $\\
\hline \rule[-4mm]{0mm}{1cm}
$(\frac{\tau}{2},\frac{\tau}{2},\frac{1}{2})$ & $6$ & $\bZ/2\bZ$ \\
 \hline \rule[-4mm]{0mm}{1cm}
$(\frac{\tau}{2},\frac{\tau}{2},\not=\frac{1}{2})$  & $4$ & $\bZ/2\bZ\times \bZ/4\bZ$ \\
 \hline \rule[-4mm]{0mm}{1cm}
$(\frac{\tau}{2},\frac{\tau+1}{2},\frac{1}{2})$  & $8$ & $\{0\}$\\
 \hline \rule[-4mm]{0mm}{1cm}
$(\frac{\tau}{2},\frac{\tau+1}{2},\not=\frac{1}{2})$  & $6$  & $\bZ/2\bZ$ \\
 \hline \rule[-4mm]{0mm}{1cm}
 $(0,\frac{1}{2},\frac{1}{2})$  & $6$ & $\bZ/2\bZ$ \\
 \hline \rule[-4mm]{0mm}{1cm}
 $(0,\frac{1}{2},\not=\frac{1}{2})$  & $4$ & $\bZ/2\bZ\times \bZ/2\bZ$\\
 \hline 
 \end{tabular}}
\caption{Hodge numbers and fundamental group of crepant resolution $Y$ of $X/(\bZ/2\bZ)$}
\label{table3}
\end{table}

\begin{table}[H]
      \centering
      \scalebox{0.74}{
        \begin{tabular}{|c|c|c|c|}
     \hline \rule[-4mm]{0mm}{1cm}
$\alpha_{A'}$ & $\beta_{A'}$ & $h^{1,1}(Y)=h^{2,1}(Y)$ & $\pi_1(Y)$ \\
\hline \rule[-4mm]{0mm}{1cm}
$(0,0,\frac{\tau'}{2})$ & $(0,0,\frac{\tau'+1}{2})$&  $7$ & $0\rightarrow \Lambda_3 \rightarrow \pi_1(Y)  \rightarrow \bZ/2\bZ\times \bZ/2\bZ \rightarrow 0 $\\
\hline \rule[-4mm]{0mm}{1cm} 

 $(0,0,\frac{\tau'}{2})$ & $(\frac{\tau}{2},\frac{\tau}{2},\frac{1}{2})$ & $6$ & $\bZ/2\bZ$ \\
\hline \rule[-4mm]{0mm}{1cm} 

 $(0,0,\frac{\tau'}{2})$ &  $(\frac{\tau}{2},\frac{\tau}{2},0)$ &$4$ & $\bZ/2\bZ\times \bZ/4\bZ$ \\
\hline \rule[-4mm]{0mm}{1cm} 

$(0,0,\frac{\tau'}{2})$ &  $(\frac{\tau}{2},\frac{\tau+1}{2},\frac{1}{2})$ & $8$ & $\{0\}$\\
\hline \rule[-4mm]{0mm}{1cm} 

$(0,0,\frac{\tau'}{2})$ &   $(\frac{\tau}{2},\frac{\tau+1}{2},0)$ & $6$ & $\bZ/2\bZ$ \\
\hline \rule[-4mm]{0mm}{1cm} 

 $(0,0,\frac{\tau'}{2})$ & $(0,\frac{1}{2},\frac{1}{2})$ &  $6$ & $\bZ/2\bZ$\\
\hline \rule[-4mm]{0mm}{1cm} 

$(0,0,\frac{\tau'}{2})$ &  $(0,\frac{1}{2},0)$ &$4$ & $\bZ/2\bZ\times \bZ/2\bZ$ \\

 $(0,0,\frac{\tau'+1}{2})$ &  $(\frac{\tau}{2},\frac{\tau}{2},\frac{1}{2})$ &$6$ & $\bZ/2\bZ$\\
\hline \rule[-4mm]{0mm}{1cm} 

 $(0,0,\frac{\tau'+1}{2})$ & $(\frac{\tau}{2},\frac{\tau}{2},0)$ &   $4$ &  $\bZ/2\bZ\times \bZ/4\bZ$\\
\hline \rule[-4mm]{0mm}{1cm} 

$(0,0,\frac{\tau'+1}{2})$ &  $(\frac{\tau}{2},\frac{\tau+1}{2},\frac{1}{2})$ & $8$ & $\{0\}$ \\
  \hline \rule[-4mm]{0mm}{1cm} 

$(0,0,\frac{\tau'+1}{2})$ &  $(\frac{\tau}{2},\frac{\tau+1}{2},0)$ & $6$ & $\bZ/2\bZ$\\

\hline \rule[-4mm]{0mm}{1cm} 

 $(0,0,\frac{\tau'+1}{2})$ &  $(0,\frac{1}{2},\frac{1}{2})$ &  $6$ & $\bZ/2\bZ$\\
\hline \rule[-4mm]{0mm}{1cm} 

$(0,0,\frac{\tau'+1}{2})$ & $(0,\frac{1}{2},0)$ & $4$ & $\bZ/2\bZ\times \bZ/2\bZ$\\
 
\hline \rule[-4mm]{0mm}{1cm}
 $(0,0,\frac{1}{2})$ &  $(\frac{\tau}{2},\frac{\tau}{2},\frac{1}{2})$ & $12$ & $\bZ/2\bZ$\\
 \hline  \rule[-4mm]{0mm}{1cm} 

 $(0,0,\frac{1}{2})$ &  $(\frac{\tau}{2},\frac{\tau}{2},\frac{\tau'}{2})$ &
  $10$ & $\bZ/2\bZ$\\
  
\hline \rule[-4mm]{0mm}{1cm} 

$(0,0,\frac{1}{2})$ & $(0,\frac{1}{2},\frac{1}{2})$ & $12$ & $\bZ/2\bZ$\\
\hline \rule[-4mm]{0mm}{1cm} 

 $(0,0,\frac{1}{2})$ &  $(0,\frac{1}{2},\frac{\tau'}{2})$ & $10$  & $\bZ/2\bZ$\\

\hline \rule[-4mm]{0mm}{1cm} 

$(\frac{\tau}{2},\frac{\tau}{2},\frac{1}{2})$ &  $(\frac{\tau}{2},\frac{\tau+1}{2},\frac{1}{2})$ & $13$ & $\{0\}$\\
 
\hline \rule[-4mm]{0mm}{1cm} 
 $(\frac{\tau}{2},\frac{\tau}{2},\frac{1}{2})$ &  $(\frac{\tau}{2},\frac{\tau+1}{2},\frac{\tau'}{2})$ & $10$ & $\{0\}$\\

\hline \rule[-4mm]{0mm}{1cm} 
 $(\frac{\tau}{2},\frac{\tau}{2},\frac{1}{2})$ &  $(\frac{\tau}{2},\frac{\tau+1}{2},\frac{\tau'+1}{2})$ &  $10$ & $\{0\}$\\
  \hline \rule[-4mm]{0mm}{1cm} 
 $(\frac{\tau}{2},\frac{\tau}{2},\frac{1}{2})$ &  $(\frac{\tau}{2},\frac{\tau+1}{2},0)$ &  $14$ & $\{0\}$ \\
\hline \rule[-4mm]{0mm}{1cm}

 $(\frac{\tau}{2},\frac{\tau}{2},0)$ & $(\frac{\tau}{2},\frac{\tau+1}{2},\frac{1}{2})$  
 & $13$ & $\{0\}$\\
 
 \hline \rule[-4mm]{0mm}{1cm} 
  $(\frac{\tau}{2},\frac{\tau}{2},0)$ & $(\frac{\tau}{2},\frac{\tau+1}{2},\frac{\tau'}{2})$ & $8$ & $\bZ/2\bZ$\\
   
\hline \rule[-4mm]{0mm}{1cm} 
 $(\frac{\tau}{2},\frac{\tau}{2},0)$& $(\frac{\tau}{2},\frac{\tau+1}{2},\frac{1+\tau'}{2})$  & $8$ & $\bZ/2\bZ$\\
 
\hline \rule[-4mm]{0mm}{1cm} 
  $(\frac{\tau}{2},\frac{\tau}{2},0)$ & $(\frac{\tau}{2},\frac{\tau+1}{2},0)$ & $10$ & $\bZ/2\bZ$\\
  
  \hline \rule[-4mm]{0mm}{1cm}
  $(\frac{\tau}{2},\frac{\tau}{2},\frac{\tau'}{2})$& $(\frac{\tau}{2},\frac{\tau+1}{2},\frac{\tau'}{2})$  & 
  $8$ & $\bZ/2\bZ$\\

  \hline \rule[-4mm]{0mm}{1cm}
   
 $(\frac{\tau}{2},\frac{\tau}{2},\frac{\tau'}{2})$ & $(\frac{\tau}{2},\frac{\tau+1}{2},\frac{1+\tau'}{2})$   & $10$ & $\bZ/2\bZ$\\
 
 \hline \rule[-4mm]{0mm}{1cm} 
 $(\frac{\tau}{2},\frac{\tau}{2},\frac{\tau'}{2})$& $(\frac{\tau}{2},\frac{\tau+1}{2},0)$  & $6$ & $\bZ/2\bZ$\\
\hline \rule[-4mm]{0mm}{1cm}
  $(\frac{\tau}{2},\frac{\tau}{2},\frac{\tau'+1}{2})$ & $(\frac{\tau}{2},\frac{\tau+1}{2},\frac{1}{2})$ & $7$ & $\{0\}$\\
 
  \hline \rule[-4mm]{0mm}{1cm} 
  $(\frac{\tau}{2},\frac{\tau}{2},\frac{\tau'+1}{2})$& $(\frac{\tau}{2},\frac{\tau+1}{2},\frac{\tau'}{2})$ & $10$ & $\bZ/2\bZ$\\

\hline \rule[-4mm]{0mm}{1cm} 
 $(\frac{\tau}{2},\frac{\tau}{2},\frac{\tau'+1}{2})$& $(\frac{\tau}{2},\frac{\tau+1}{2},\frac{1+\tau'}{2})$  &  $8$ & $\bZ/2\bZ$\\

  \hline \rule[-4mm]{0mm}{1cm} 
 $(\frac{\tau}{2},\frac{\tau}{2},\frac{\tau'+1}{2})$& $(\frac{\tau}{2},\frac{\tau+1}{2},0)$ & $6$ & $\bZ/2\bZ$\\
\hline 
\end{tabular}}
\caption{Hodge numbers and fundamental group of crepant resolution $Y$ of $X/(\bZ/2\bZ)^2$}
\label{table10}
\end{table}

\begin{table}[H]
    \centering
          \scalebox{0.78}{
\begin{tabular}{|c|c|c|c|c|}
\hline \rule[-4mm]{0mm}{1cm}
$\alpha_{A'}$ & $\beta_{A'}$ & $\gamma_{A'}$ & $h^{1,1}(Y)=h^{2,1}(Y)$ & $\pi_1(Y)$\\
\hline \rule[-4mm]{0mm}{1cm}
$(0,0,\frac{\tau'}{2})$ & $(0,0,\frac{1}{2})$ &  $(\frac{\tau}{2},\frac{\tau}{2},\frac{1}{2})$ & $12$ & $\bZ/2\bZ$ \\

\hline \rule[-4mm]{0mm}{1cm}
$(0,0,\frac{\tau'}{2})$ & $(0,0,\frac{1}{2})$ &  $(\frac{\tau}{2},\frac{\tau+1}{2},\frac{1}{2})$ & $17$ & $\{0\}$\\

\hline \rule[-4mm]{0mm}{1cm}
$(0,0,\frac{\tau'}{2})$ & $(0,0,\frac{1}{2})$ &  $(0,\frac{1}{2},\frac{1}{2})$ & $12$ & $\bZ/2\bZ$ \\

\hline \rule[-4mm]{0mm}{1cm}
$(0,0,\frac{\tau'}{2})$ & $(\frac{\tau}{2},\frac{\tau}{2},\frac{1}{2})$ &  $(\frac{\tau}{2},\frac{\tau+1}{2},\frac{1}{2})$ 
&  $13$ & $\{0\}$ \\

\hline \rule[-4mm]{0mm}{1cm}
$(0,0,\frac{\tau'}{2})$ & $(\frac{\tau}{2},\frac{\tau}{2},\frac{1}{2})$ &  $(\frac{\tau}{2},\frac{\tau+1}{2},0)$ & $14$ & $\{0\}$ \\

\hline \rule[-4mm]{0mm}{1cm}
$(0,0,\frac{\tau'}{2})$ & $(\frac{\tau}{2},\frac{\tau}{2},0)$ &  $(\frac{\tau}{2},\frac{\tau+1}{2},\frac{1}{2})$ & $13$ & $\{0\}$\\

\hline \rule[-4mm]{0mm}{1cm}
$(0,0,\frac{\tau'}{2})$ & $(\frac{\tau}{2},\frac{\tau}{2},0)$ &  $(\frac{\tau}{2},\frac{\tau+1}{2},0)$ &  $9$ & $\bZ/2\bZ$\\
\hline \rule[-4mm]{0mm}{1cm}

$(0,0,\frac{\tau'+1}{2})$ & $(\frac{\tau}{2},\frac{\tau}{2},\frac{1}{2})$ &  $(\frac{\tau}{2},\frac{\tau+1}{2},\frac{1}{2})$ & $13$ & $\{0\}$ \\

\hline \rule[-4mm]{0mm}{1cm}
$(0,0,\frac{\tau'+1}{2})$ & $(\frac{\tau}{2},\frac{\tau}{2},\frac{1}{2})$ &  $(\frac{\tau}{2},\frac{\tau+1}{2},0)$ &  $14$ & $\{0\}$\\

\hline \rule[-4mm]{0mm}{1cm}
$(0,0,\frac{\tau'+1}{2})$ & $(\frac{\tau}{2},\frac{\tau}{2},0)$ &  $(\frac{\tau}{2},\frac{\tau+1}{2},\frac{1}{2})$ & $13$ & $\{0\}$ \\

\hline \rule[-4mm]{0mm}{1cm}
$(0,0,\frac{\tau'+1}{2})$ & $(\frac{\tau}{2},\frac{\tau}{2},0)$ &  $(\frac{\tau}{2},\frac{\tau+1}{2},0)$ & $10$  & $\bZ/2\bZ$ \\
\hline \rule[-4mm]{0mm}{1cm}

$(0,0,\frac{1}{2})$ & $(\frac{\tau}{2},\frac{\tau}{2},\frac{1}{2})$ &  $(\frac{\tau}{2},\frac{\tau+1}{2},\frac{1}{2})$ 
&  $27$ & $\{0\}$ \\
\hline \rule[-4mm]{0mm}{1cm}
$(0,0,\frac{1}{2})$ & $(\frac{\tau}{2},\frac{\tau}{2},\frac{1}{2})$ &  $(\frac{\tau}{2},\frac{\tau+1}{2},\frac{\tau'}{2})$ &  $20$ & $\{0\}$ \\

\hline \rule[-4mm]{0mm}{1cm}
$(0,0,\frac{1}{2})$ & $(\frac{\tau}{2},\frac{\tau}{2},\frac{\tau'}{2})$ &  $(0,\frac{1}{2},0)$ & $18$ & $\{0\}$ \\

\hline \rule[-4mm]{0mm}{1cm}
$(0,0,\frac{1}{2})$ & $(\frac{\tau}{2},\frac{\tau}{2},\frac{\tau'}{2})$ &  $(0,\frac{1}{2},\frac{\tau'}{2})$ &  $15$ & $\{0\}$\\

\hline
\end{tabular}}
\caption{Hodge numbers and fundamental group of crepant resolution $Y$ of $X/(\bZ/2\bZ)^3$}
\label{table12}
\end{table}


\begin{table}[H]
\centering
          \scalebox{0.78}{
\begin{tabular}{|c|c|c|}
\hline \rule[-4mm]{0mm}{1cm}
$\Upsilon=\text{Aut}(X)$ & $h^{1,1}(Y)=h^{2,1}(Y)$ & $\pi_1(Y)$\\
 \hline \rule[-4mm]{0mm}{1cm}
& $27$  &  $\{0\}$ \\
\hline
\end{tabular}}
\caption{Hodge number and fundamental group of crepant resolution $Y$ of $X/(\bZ/2\bZ)^4$}
\label{table14}
\end{table}

As evidence we show explicitly one computation for $\pi_1(X/\Upsilon)$.
Let $\Upsilon=\langle \alpha_X\rangle$ and $\alpha_X$ is induced by $\alpha_{A'}(\ul{z})=(z_1,z_2,z_3+\frac{1}{2})$. Then $X/\Upsilon=\bC^3/\Gamma$ where $\Gamma$ is a finite extension of $\pi_1(A')$ by $\langle H, \alpha_{A'}\rangle$. By Corollary \ref{fund-quotient} : $\pi_1(X/\Upsilon)\simeq\dfrac{\Gamma}{F_{\Gamma}}$. To compute $F_{\Gamma}$ we
use Remark \ref{remark: fixed locus alphaX} as explained in Remark \ref{computation-fund-group}, we find 
$F_{\Gamma}=\{\lambda\ol{\alpha_{A'}} \ol{s}^2 \ol{r}^2,\lambda\ol{\alpha_{A'}} \ol{s}^2 \ol{r}^2 \mid \lambda(\ul{z})=(z_1+t_1,z_2+t_2,z_3) \text{ with } t_i\in \Lambda_i\}$. Thus, we get:
$$\pi_1(X/\Upsilon)\simeq\dfrac{\Gamma}{F_{\Gamma}}=\langle \ol{r},\ol{s},\ol{\alpha_{A'}}, \Lambda_3 \mid \ol{r}^2=(\ol{\alpha_{A'}})^{-1}, \ol{s}^2=\alpha_{A'}^2\rangle. $$
We observe that $\Lambda_3'''=\dfrac{1}{2}\bZ\oplus \bZ\subset \Lambda_3$ is the maximal abelian and normal subgroup of finite index in $\dfrac{\Gamma}{F_{\Gamma}}$ and the images of $\ol{r}$ and $\ol{s}$ in $\dfrac{\Gamma/F_{\Gamma}}{\Lambda_3'''}$ define two automorphisms of order $2$ which commute, thus we obtain the following characterization of $\pi_1(X/\Upsilon)$:
\begin{center}
    \begin{tikzcd}
        0\longrightarrow \Lambda_3''' \longrightarrow \pi_1(X/\Upsilon) \longrightarrow \bZ/2\bZ \times \bZ/2\bZ \longrightarrow 0.
    \end{tikzcd}
\end{center}

 \bibliographystyle{plain}
 \bibliography{bibliography}

\end{document}